\documentclass[11pt]{amsart}      
\usepackage[a4paper,hmargin={2.5cm,2.5cm},vmargin={2.5cm,2.5cm}]{geometry}
\usepackage{geometry}               

\usepackage{graphicx}
\usepackage[dvipsnames]{xcolor}
\usepackage{amssymb,amsmath,amsthm,amsfonts}
\allowdisplaybreaks
\usepackage[english]{babel}
\usepackage[utf8]{inputenc}
\usepackage[colorlinks]{hyperref}
\hypersetup{linkcolor=blue,citecolor=blue,filecolor=black,urlcolor=blue}
\usepackage{comment}
\usepackage{mathtools}
\usepackage{bigints}
\usepackage{dsfont}
\usepackage{tikz}
\usepackage{caption}

\newtheorem{proposition}{Proposition}[section]

\newtheorem{theorem}[proposition]{Theorem}

\newtheorem{lemma}[proposition]{Lemma}

\theoremstyle{definition}

\newtheorem{definition}[proposition]{Definition}

\theoremstyle{remark}

\newtheorem{remark}[proposition]{Remark}

\numberwithin{equation}{section}

\newcommand{\R}{{\mathbb{R}}}

\title[Monotonicity of non-negative solutions in cylindrical domains]{Monotonicity of non-negative solutions of quasilinear elliptic equations in a cylindrical domain}

\author{Luigi Montoro, Luigi Muglia, Berardino Sciunzi, Giuseppe Spadaro}
\address{Department of Mathematics and Computer Science, UNICAL, Rende, CS, Italy}
\email{luigi.montoro@unical.it, luigi.muglia@unical.it}
\email{berardino.sciunzi@unical.it, giuseppe.spadaro@unical.it}
\date{\today}

\begin{document}


\begin{abstract}
We consider weak solutions to $p$-Laplace equations in  cylindrical domains under mixed homogeneous Dirichlet-Neumann boundary conditions. We assume that the right-hand side is positive and locally Lipschitz continuous and we prove that any positive  solution is monotone increasing in the $x_N$ direction for any $p>1$. As an application we prove  that  solutions to Allen-Cahn type equations are one-dimensional as well as a Liouville type result for Lane-Emden type equations.
\end{abstract}%
\maketitle

\noindent

{\footnotesize \textbf{2020 Mathematics Subject classification: }35J62, 35B06, 35B35, 35B53}.

{\footnotesize 

}

{\footnotesize \textbf{Keywords.} Quasilinear problem, qualitative properties in cylindrical domains, classification result, Liouville-type theorem}.

{\footnotesize 
}

\section{Introduction and statement of the main results}

We consider the problem

\begin{equation}\tag{$\mathcal{P}$}\label{main_problem}
\begin{cases}
-\Delta_{p} u = f(u) & \text{in } C_+\\
u > 0  & \text{in } C_+\\
\partial_{\nu} u = 0  & \text{on } \partial{\Omega} \times (0,+\infty)\\
u=0 & \text{on } \Omega \times \{0\},\\
\end{cases}
\end{equation}
where $C_+$ is a cylinder in $\mathbb{R}^N$, with $N \geq 2$, namely:
\begin{equation*}
C_+:=\Omega \times (0,+\infty),
\end{equation*}
and $\Omega \subset \R^{N-1}$ is a bounded domain, with a smooth boundary.  We shall prove that any positive solution is  monotone increasing and we will also point out some consequences in the case of interesting leading examples. \\

\

In what follows we denote by $(x',x_N)$, with $x' \in \Omega$ and $x_N \in (0,+\infty)$ a generic point in $C_+$, moreover we assume that the source term $f$ satisfies:
\begin{equation}\label{Hp}\tag{H}
\begin{split}
&\text{$f:[0,+\infty)\to \mathbb{R}$ is positive, that is $f(s)>0$ for $s > 0$, }\text{and $L$-locally Lipschitz continuous.}
\end{split}
\end{equation}
It is well known that solutions of $p$-Laplace equations are not classical also in the case of smooth domains (see \cite{DB,KM,Lib0,tolk}), thus the equation \eqref{main_problem} has to be understood in the weak sense. For this reason we next recall the following definition of mixed space: let $A$ be an open set in $\mathbb{R}^N$, and $D$ is a closed subset of $\overline A$. Define
\begin{equation}\label{test}
 W^{1,p}_D(A):=\text{the closure of} \ \{\varphi|_A:\varphi\in C_c^\infty(\mathbb{R}^N), supp(\varphi)\cap {D}= \emptyset\} \mbox{ in } {W^{1,p}(A)}.
\end{equation}
Let 
$$
 C_\sigma:=\Omega \times (0,\sigma).
$$
\begin{definition} A measurable function $u:\overline{C}_+\to \mathbb{R}$ is said to be a weak solution to \eqref{main_problem} if, for every $\sigma>0$, $u\in W^{1,p}_{\Omega}(C_\sigma)\equiv W^{1,p}_{\Omega\times\{0\} }(C_\sigma)$ and 
\begin{equation*}
\int_{C_+}|\nabla u|^{p-2}\nabla u\cdot\nabla \varphi \ dx=\int_{C_+}f(u)\varphi \ dx,
\end{equation*}
for every $\varphi:=\varphi|_{C_+}$ such that $\varphi \in C^\infty_c(\mathbb{R}^N_+)$.
\end{definition}
\noindent
Moreover, according to the regularity results in \cite{Lib0}, we assume that $u \in C_{loc}^{1,\alpha}(\overline C_+)$ and fulfills the equation in the weak sense. Actually, since in our case the domain is not of class $C^1$, the $C_{loc}^{1,\alpha}$ regularity does not follow directly by \cite{Lib0} and a standard reflection argument, w.r.t. the base of the cylinder $\overline \Omega \times \{0\}$, is needed. This is indeed possible due to the Dirichlet datum.\\

\

Our main goal is to study the monotonicity of solutions of \eqref{main_problem} in the $x_N$ direction.
To this end, a crucial step is to establish a weak comparison principle in $C_+$ for nonnegative solutions of the $p$-Laplace equation subject to mixed Dirichlet-Neumann boundary conditions. This will enable the application of the moving plane method that goes back to \cite{A,S}.\\
The moving plane method serves as a fundamental tool to prove monotonicity and symmetry properties of solutions to general PDEs. Specifically, in the semilinear case $p=2$, this study has been done in the seminal papers \cite{BN,GNN}, which addressed the problem in bounded domains. Concerning the case of unbounded domains, the main examples are provided by the whole space $\mathbb{R}^N$ and the half-space $\mathbb{R}_+^N$. For the case of the whole space we refer the reader to \cite{CGS,GNN2} and for the case of the half-space we refer to \cite{BCN,BCN2,BCN3,BF,DAN,F,FV}. Moreover, we mention some recent results in this direction \cite{BG1,BG2,BG3,BGW}.
Our result, restricted to the semilinear case $p=2$, is very much related to the interesting results in  \cite{CWY}.\\
Here the degenerate nature of the $p$-Laplacian makes the analysis substantially more delicate. Initial results concerning the case $1<p<2$ in bounded domains were established in \cite{DP}. The study of qualitative properties of solutions, in the case $p>2$ and in bounded domains, requires the use of weighted Sobolev spaces and, in particular, the use of weighted Poincar\'e type inequalities with weight $\rho = |\nabla u|^{p-2}$, see \cite{D-S-2,D-S}. Concerning the case of the whole space $\mathbb{R}^N$ we refer the reader to \cite{CFR,  CPP,CL,  CCG, CFP,DMMS,Sciunzi, V}. Problems in half-spaces have been studied by several authors, in particular the singular case $1<p<2$ in \cite{FMRS,FMS2} under the assumption of positive, locally Lipschitz continuous right-hand sides. In the degenerate case $p>2$, which presents significantly greater difficulties, the first result was obtained in \cite{FMS} under the restriction $2<p<3$ and power type nonlinearities are considered. Moreover, in \cite{FMS3} the authors removed the condition $2<p<3$ considering, in addition, a larger class of nonlinearities.\\

\
 Concerning problems in cylindrical domains involving the $p$-Laplacian, we present here the first result, up to our knowledge. The readers will appreciate a strong correlation with the regularity theory developed in \cite{MMS} regarding the optimal second order regularity up to the boundary. This issue turns out to be crucial in order to obtain the right weighted Sobolev embedding, weak and strong comparison principles and boundary Harnack inequalities. All these tools shall play a crucial role in the application of the moving plane procedure. Let us stress the fact that, in the semilinear case, some of the technical tools that we have to recover, are classical. The presence of the $p$-Laplace operator causes that we have to face also a geometrical decomposition argument that is interesting in itself.

Let us state our main result:
\begin{theorem}\label{Monotonicity}
Let $\Omega \subset \mathbb{R}^{N-1}$ be a bounded smooth  domain, $u \in C^{1,\alpha}_{loc}(\overline{C}_+)$ be a weak solution of \eqref{main_problem} and assume that \eqref{Hp} holds. Then, for any $p>1$, $u$ is monotone increasing in the $x_N$ direction.
\end{theorem}
\noindent

The monotonicity result in Theorem \ref{Monotonicity} is quite general. A first interesting case arise when considering  solutions to Allen-Cahn type equations, namely for $f(u)=u(1-u^2)^{p-1}$.
In Section \ref{ACtype}, in fact,  we study  solutions of:
\begin{equation}\label{AC_INTRO}
\begin{cases}
-\Delta_{p} u = u(1-u^2)^{p-1} & \text{in } C_+\\
0< u < 1  & \text{in } C_+\\
\partial_{\nu} u = 0  & \text{on } \partial{\Omega} \times (0,+\infty)\\
u=0 & \text{on } \Omega \times \{0\}.\\
\end{cases}
\end{equation}
In particular, exploiting also our   monotonicity result, we prove that the solutions of \eqref{AC_INTRO} exhibit 1-D symmetry, namely we have the following:
\begin{theorem}\label{AC_teo}
Let $u \in C^{1,\alpha}_{loc}(\overline{C}_+)$ be a weak solution of \eqref{AC_INTRO}, then 
\begin{equation*}
u(x',x_N)=u(x_N)= \tanh\left(\frac{x_N}{(p-1)^{\frac{1}{p}}2^{\frac{1}{p}}}\right).
\end{equation*}
\end{theorem}

In cases when such a rigidity result cannot be expected, the monotonicity of the solutions induces, in any case, the stability of the solutions. Such an information has to be understood in the right meaning in our context. 
An important consequence is that this can furthermore be employed in the derivation of Liouville-type theorems, see \cite{DFP,DSS,F2,FVS,FSVuono}. 
In particular, in the case $p\geq 2$, our monotonicity result holds for Lane-Emden type equations, namely in the case $f(u)=u^q$ with $q > p-1$. This enables us to prove the following theorem.
\begin{theorem}\label{Stabilità}
Let $p\geq 2$ and  $u\in C_{loc}^{1,\alpha}(\overline C_+)$ be a solution of
\begin{equation}\label{main_problem_S}
\begin{cases}
-\Delta_{p} u = u^q & \text{in } C_+\\
u \geqslant 0  & \text{in } C_+\\
\partial_{\nu} u = 0  & \text{on } \partial{\Omega} \times (0,+\infty)\\
u=0 & \text{on } \Omega \times \{0\},\\
\end{cases}
\end{equation}
with $q > p-1$. Then $u\equiv 0$.
\end{theorem}

For the reader's convenience we provide a description  of the structure of the paper:
\begin{itemize}
	\item[-] As remarked above, a crucial issue in our proof is the application of \textit{weighted Sobolev inequalities}
	that are based on potential estimates. This will require a fine cube decomposition\textit{} that we
	deduce in Section \ref{notations}. In Section \ref{notations} we also prove the \textit{boundary type Harnack inequality} basing on a weighted Sobolev inequality and the \textit{Moser iteration scheme}. 
	\item[-] Section \ref{monomono} is devoted to the proof of Theorem \ref{Monotonicity} that is based on a refined version of the \textit{moving plane procedure}. To exploit this technique we need to recover \textit{weak and strong comparison principles} and, furthermore, a suitable version of the \textit{Hopf boundary lemma}. 
	\item[-] In Section \ref{liouville} we will prove the \textit{rigidity result} Theorem \ref{AC_teo} and the \textit{Liouville type result} Theorem \ref{Stabilità}. The proof of Theorem \ref{AC_teo} benefit of the monotonicity result since this allows to study the limiting profile and, consequently, to exploit the \textit{sweeping principle}. The Liouville type result follows actually by the fact that monotone solutions are actually \textit{stable solutions} (in a suitable meaning here). The technique that we will use in this issue is the one introduced in \cite{F2}, as developed in \cite{DFSV} for the quasilinear case.
\end{itemize}

\section{Cube decomposition of $\Omega$ and Harnack inequalities}\label{notations}
\noindent We prove here two results that will be used in the proof of Theorem \ref{Monotonicity}.\\ Generic fixed numerical constants will be denoted by $C$ (with subscript in some case) and will be allowed to vary within a single line or formula. We also denote by $|A|$ the Lebesgue measure of the set $A$.\\

Later one we will base the proof of comparison principles on the use of a weighted Poincar\'e inequality. The lack of the zero Dirichlet datum all over the boundary, causes that we need to use a suitable version arising from potential estimates. We shall be therefore reduced to use a domain decomposition that we require to have specific geometric properties. This is why we start with the following:

\begin{theorem}\label{main_CUBE}
Let $\Omega$ be a bounded open set in $\mathbb{R}^{N-1}$, with smooth boundary. Then, there exists $\delta_0 = \delta_0(N,\Omega)>0$ such that for every $0 < \delta \leq \delta_0$, there exist $M=M(\delta)$ and open sets $A_i\subset \mathbb{R}^{N-1}$, $i=1\ldots, M$,  such that $\Omega$ is decomposed as
\begin{equation*}
\Omega = \bigcup_{i=1}^M A_i,
\end{equation*}
where $A_i$ satisfying, for each $i=1,...,M$,
\begin{equation*}
\begin{split}
&\operatorname{diam} (A_i ) \sim \delta,\\
&|A_i| \sim \delta^{N-1},
\end{split}
\end{equation*}
namely that there exist two constants $c(N,\Omega)$, $C(N,\Omega)$ such that:
\begin{equation*}
c(N,\Omega) \delta \leq \operatorname{diam} (A_i) \leq C(N,\Omega) \delta,
\end{equation*}
\begin{equation*}
c(N,\Omega) \delta^{N-1} \leq |A_i| \leq C(N,\Omega) \delta^{N-1}.
\end{equation*}
\end{theorem}
%
Before providing the proof of Theorem \ref{main_CUBE} let us consider $\delta>0$ and $Q$ a cube in $\mathbb{R}^{N-1}$ such that $\Omega \subset Q$. Consider a partition of $Q$ given by a family of finite cubes $(Q_{\delta,i})_{i=1,\ldots,\bar{N}}$ of edge $\delta$, namely
\begin{equation}\label{insiemi}
 \overline \Omega \subset \overline Q = \bigcup_{i=1}^{ \overline N} \overline Q_{\delta,i}.
\end{equation}

We divide the family of $\{Q_{\delta,i}\}_{i}$ into two subfamilies: cubes $I_{\delta,i}\in\{Q_{\delta,i}\}_{i}$ such that $\overline{I_{\delta,i}} \subset \Omega$ and cubes $I_{\delta,i}^{\partial}\in\{Q_{\delta,i}\}_{i}$ such that $\overline{I_{\delta,i}^{\partial}} \cap \partial \Omega \neq \emptyset$. We set,

\begin{equation*}
I : = \bigcup_{i=1}^{N^I} I_{\delta,i} \quad \text{and} \quad I^\partial : = \bigcup_{i=1}^{N^\partial} I_{\delta,i}^\partial, \quad \text{with } N^I + N^\partial \leq \overline N.
\end{equation*}

Moreover, for all fixed $i$ and all given $I_{\delta,i}^\partial$, let us denote by  $Q_{3\delta,i}^G$ the cube (union of $Q_{\delta,i}$) with the same center of $I_{\delta,i}^\partial$ and edge $3\delta$.
The proof of Theorem \ref{main_CUBE} will be based also  on the following:
\begin{lemma}\label{esis_delta_0}
There exists $\delta_0 > 0$ such that for every $0 < \delta \leq \delta_0$ and for every $I_{\delta,i}^\partial \in I^\partial$, there exists a cube $Q_{\delta,j}\subset Q_{3\delta,i}^G$ in the partition \eqref{insiemi}, such that:
\begin{equation}\label{star}
\begin{split}
&\operatorname{diam} (Q_{\delta,j} \cap \Omega) > \frac{\sqrt{N-1}}{2} \delta,\\
&|Q_{\delta,j}\cap \Omega| >\left (\frac{\delta}{2}\right)^{N-1}.
\end{split}
\end{equation}
\end{lemma}
\begin{proof}
Let $x_0 \in \partial \Omega$ and assume that there exists a neighborhood $V_{x_0}$ such that $V_{x_0} \cap \partial \Omega$ is flat. If the cube centered at $x_0$ satisfies \eqref{star}, we are done. Otherwise, after possibly redefining $\delta$ to be smaller, and using that $\Omega$ is a smooth domain, we can find $Q_{\delta,i} \subset Q_{3\delta,i}^G$ entirely contained in $\Omega$, hence satisfying \eqref{star}.\\
If $\partial \Omega$ is not flat, the result follows from the fact that, since $\partial \Omega$ is smooth, we can locally approximate the boundary by its tangent plane (see Figure \ref{figura}).
\begin{figure}[ht]
\centering
\begin{tikzpicture}[scale=1]

  \def\n{5}     
  \def\s{1}     

  \foreach \i in {0,1,...,\n} {
    \draw[thick] (\i*\s, 0) rectangle ++(\s, \s);
    \draw[thick] (\i*\s, 1) rectangle ++(\s, \s);
    \draw[thick] (\i*\s, -1) rectangle ++(\s, \s);
  }

  \def\m{0.3}   
  \def\c{0.2}   

  \draw[thick]
    ({-2}, { \m*(-1) + \c}) -- ({\n*\s + 3}, { \m*(\n*\s + 1) + \c });
  \draw[red, thick, domain=-2:8, samples=200, smooth, variable=\x]
    plot ({\x}, {\m*\x + \c + 0.1*sin(2*\x r)-0.35});
 \node[black] at (7, {0.3*8 + 0.2 + 0.2*sin(2*8 r) + 0.1}) {$\partial \Omega \cap U$};
 \node[black] at (4.25, {0.3*8 + 0.2 + 0.2*sin(2*8 r) - 1.1}) {$x_0$};
 \fill[black] (4.25, {0.3*8 + 0.2 + 0.2*sin(2*8 r) - 1.33}) circle (1pt);
\end{tikzpicture}
\caption{}
\label{figura}
\end{figure}
\\
For every $x_0 \in \partial \Omega$ there exists a neighborhood $U \subset \mathbb{R}^{N-1}$ of $x_0$ such that, up to rotation,
\begin{equation*}
\partial \Omega \cap U = \{(x',x_{N-1}) \in U : x_{N-1}=\varphi(x')\},
\end{equation*}
where $x'=(x_1,...,x_{N-2})$ and $\varphi : \mathbb{R}^{N-2} \rightarrow \mathbb{R}$ has the same regularity of $\partial\Omega$.\\
By Taylor expansion,  there exists $\delta_0 > 0$ small enough, such that for
$|x' - x'_0| \leq \delta_0$, 
\begin{equation*}
x_{N-1} = \varphi(x')= \varphi(x'_0) + \nabla \varphi(x'_0) \cdot (x'-x'_0) + O(\delta^2_0),
\end{equation*}
where, using the compactness of $\partial \Omega$, the choice of $\delta_0$ does not depend on $x_0$.
Thus, if $\partial \Omega$ is not perfectly flat, it is still nearly flat; hence, using the same argument as in the flat case above, we obtain \eqref{star}, since the error we commit is negligible.

\end{proof}
\noindent
Let us now call $B$ the set of the above cubes $Q_{\delta,j}$  that satisfy \eqref{star}, namely:
\begin{equation*}
B:= \{Q_{\delta,j} : \text{fulfills the conditions } \eqref{star}\}.
\end{equation*}
We are now in position to prove Theorem \ref{main_CUBE}.
\begin{proof}[Proof of Theorem \ref{main_CUBE}.] The proof is based on an iterative technique divided in three steps:\\
\textbf{Step $1$.} We start the procedure by taking a cube in the set $I^\partial$, namely $I_{\delta,1}^\partial$. If $I_{\delta,1}^\partial \in B$, then define the cube $\tilde I_{\delta,1}^\partial = I_{\delta,1}^\partial$. If, instead, $I_{\delta,1}^\partial \notin B$, then by Lemma \ref{esis_delta_0}, there exists $Q_{\delta,j} \in B \cap Q_{3\delta,1}^G$. In this case, we define $\tilde I_{\delta,1}^\partial = Q_{\delta,j} \cup C_1$, where:
\begin{equation*}
C_1 : = \bigcup_{i=1}^{N^\partial} \{ I_{\delta,i}^\partial : I_{\delta,i}^\partial \subset Q_{3\delta,1}^G,  I_{\delta,i}^\partial \notin B\}.
\end{equation*}
Notice that with this construction, we have that:
\begin{equation*}
\begin{split}
&\operatorname{diam}(\tilde I_{\delta,1}^\partial) \sim \delta,\\
&|\tilde I_{\delta,1}^\partial| \sim \delta^{N-1}.
\end{split}
\end{equation*}
\textbf{Step $2$.} Let $I_{\delta,k}^\partial \in I^\partial$, such that $I_{\delta,k}^\partial \cap \tilde I_{\delta,1}^\partial = \emptyset$. If $I_{\delta,k}^\partial \in B$, then set $\tilde I_{\delta,2}^\partial = I_{\delta,k}$.\\ 
Alternatively, if  $I_{\delta,k}^\partial \notin B$, there exists $Q_{\delta,j} \in B \cap Q_{3\delta,k}^G$ and we define $\tilde I_{\delta,2}^\partial = Q_{\delta,j} \cup C_2$, where
\begin{equation*}
C_2 := \bigcup_{i=1}^{N^\partial}\{ I_{\delta,i}^\partial : I_{\delta,i}^\partial \subset Q_{3\delta,2}^G, I_{\delta,i}^\partial \notin B, I_{\delta,i}^\partial\cap \tilde I_{\delta,1}^\partial=\emptyset \}.
\end{equation*}
\textbf{Step $3$.} Repeat Step 2 until we can construct a set  $\tilde I_{\delta,s}^\partial$ such that there exists $I_{\delta,k}^\partial\in I^{\partial}$  for which
\begin{equation*}
I_{\delta,k}^\partial \cap \left ( \tilde I_{\delta,s}^\partial \cup ... \cup \tilde I_{\delta,1}^\partial \right ) \neq \emptyset,
\end{equation*}
for $k=1,...,N^\partial$.\\
Finally we get a family of sets that satisfy the properties:
\begin{equation*}
\begin{split}
&\operatorname{diam}(\tilde I_{\delta,j}^\partial) \sim \delta,\\
&|\tilde I_{\delta,j}^\partial| \sim \delta^{N-1},
\end{split}
\end{equation*}
for any $j=1,...,s$; with $s \leq N^\partial$.\\
Therefore, up to a reindexing of the sets in the family \eqref{insiemi}, we obtain the statement of our theorem by defining the sequence of sets ${A_i}$ as follows:
\begin{equation*}
A_i =
\begin{cases}
\Omega \cap \tilde I_{\delta,i}^\partial \text{ for } i=1,...,s,\\
Q_{\delta,i}\mbox{ for the remaining indices in }\eqref{insiemi}.
\end{cases}
\end{equation*}
\end{proof}

To conclude this section we prove a boundary type Harnack inequality for our problem on the lateral surface. Here we need to exploit some notations and results involving Fermi coordinates used in \cite{MMS}. Since $\Omega \subset \mathbb{R}^{N-1}$ is a bounded domain with smooth boundary, then the lateral surface of the cylinder $C_+$ is a smooth domain in $\mathbb{R}^N$. Therefore, for any $\bar x \in \partial \Omega \times (0,+\infty)$ there exist $r>0$ and a ball $B_r(\bar x)$ such that $(\partial \Omega \times (0,+\infty)) \cap B_r(\bar x)$ can be represented by a smooth map $\gamma : A \subset \mathbb{R}^{N-1} \rightarrow \mathbb{R}^N$. For any $\delta > 0$, we set
\begin{equation*}
    B_{\delta}^+ : =\{x\in B_r(\bar x): d(x,\partial \Omega \times (0,+\infty)) < \delta\}.
\end{equation*}
By \cite{U}, there exists $\delta_0>0$ such that $B_{\delta_0}^+$ has the unique nearest point property, i.e. for any $y \in B_{\delta_0}^+$ there exists a unique $x \in B_r(\bar x)$ such that $d(x,y) = d(y,\partial \Omega \times (0,+\infty))$.\\
Let now $\bar \delta \leq \delta_0$ such that $\bar \delta = d(\Omega,\Omega')$ and define the flattening operator $\Phi : \mathbb{R}^N \rightarrow \mathbb{R}^N$ as follows: for all $x \in B_{\bar \delta}^+$, let $\bar x = \gamma(y_1,...,y_{N-1}) \in \partial \Omega \times  (0,+\infty)$ and $\eta$ be the inward normal vector in $\bar x$. Then,
\begin{equation*}
    x = \Phi(y):= \gamma(y_1,...,y_{N-1}) + y_N \eta.
\end{equation*}
Notice that in the new local-coordinate the hyperplane $\{y_N=0\}$ represents, locally, the points of $\partial \Omega \times  (0,+\infty)$. Moreover, we denote by $\mathcal{B}_{ft}^+$ the set of the $y$-space such that
\begin{equation*}
    B_{\bar \delta}^+ = \Phi(\mathcal{B}_{ft}^+).
\end{equation*}

Let $p>2$ and $\bar{x}\in \partial \Omega\times(0,+\infty)$  and let $\delta>0$ such that $7\delta<d(\bar{x},\partial \Omega\times\{0\})$. In what follows, we will denote by  $B_{C}(\bar{x},\sigma):=B(\bar{x},\sigma)\cap C_+$, $\sigma>0$.
 We exploit the Moser iteration technique to derive a weak Harnack comparison inequality. 

\begin{theorem}\label{WHarna} Let $p>2$ and assume that $f(u)$ fulfills  $(H)$.  Let $u\,\in\,C^{1,\alpha}_{loc}(\overline{C}_+)$ be a weak solution of  \eqref{main_problem} in $C_+$. 
Let $v\in C^{1,\alpha}_{loc}(\overline{B_C(\bar{x},7\delta)})$ be such that
$$
-\Delta_{p} u - f(u) \leq -\Delta_{p} v-f(v)\qquad \mbox{in } B_C(\bar{x},7\delta).
$$
Suppose that 
 $$
 u\leq v \quad \text{in}\quad B_C(\bar{x},6\delta) .
$$
Define
\[
\frac{1}{\bar{ 2}_p}=\frac{1}{2}-\frac{1}{N}+\frac{p-2}{p-1}\,\frac{1}{N}.
\]
Then for every
\begin{equation}\nonumber
0<s<\dfrac{\bar 2_p}{2},
\end{equation}
there exists a positive constant $\tilde{C}=\tilde{C}(p, L, \delta,\|v\|_{L^{\infty}},  \|\nabla u\|_{L^{\infty}},   \|\nabla v\|_{L^{\infty}})$ such that
\begin{equation*}
\|v-u\|_{L^s(B_C(\bar{x},2\delta))}\leq \tilde{C}\inf_{B_C(\bar{x},\delta)}(v-u).
\end{equation*}
The same result holds if $v$ is a solution and $u$ is a subsolution. 
\end{theorem}
\begin{proof}
The method is strongly based on the one developed by Trudinger \cite{Tru} and the notation follows that used in \cite{MeMoSc} and \cite{Mo} (see also \cite{D-S}). Hence, we prove only the main inequalities, referring the reader to \cite{MeMoSc, Mo} for the computation of the iterating scheme. For $r\neq0$, and $h$ positive, we define 
$$
 \phi(r,\sigma,h):=\left(\int_{B_{C}(\bar{x},\sigma)}|h|^r \,  dx\right)^\frac1r.
$$ 
Referring to \cite{GT}, we recall that
\begin{eqnarray*}
 \phi(p,\sigma,h)\to \sup_{B(\bar{x},\sigma)}|h|, \qquad p\to +\infty,\\
  \phi(p,\sigma,h)\to \inf_{B(\bar{x},\sigma)}|h|, \qquad p\to -\infty.\\
\end{eqnarray*}
We proceed by proving the following steps:
\begin{itemize}
\item[Stp(i):] for any fixed $r_0>0$ and $w_\tau:=v-u+\tau$, there exists a positive constant $\tilde{C}=\tilde{C}(L, p,\delta,\|v\|_{\infty},  \|\nabla u\|_{\infty},   \|\nabla v\|_{\infty})$, such that
$$
 \phi(-\infty,\delta,w_\tau)\geq \tilde{C}\phi\left(-r_0,\frac52\delta,w_\tau\right).
$$
\item[Stp(ii):] there exists $r_0>0$ such that, for a suitable $\tilde{C}=\tilde{C}(L, p,\delta,\|v\|_{\infty},  \|\nabla u\|_{\infty},   \|\nabla v\|_{\infty})$ positive, the following estimate holds
$$
 \phi\left(-r_0,\frac52\delta,w_\tau\right)\geq \tilde{C}\phi\left(r_0,\frac52\delta,w_\tau\right).
$$
\end{itemize}
With Stp(i) and Stp(ii) in force, the proof follows by Moser’s iterative method; for details, we refer the reader to \cite[(3.28), p. 1084]{MeMoSc}–\cite[(3.35), p. 1086]{MeMoSc}.

\emph{Proof of Stp(i).} Let $\varphi\in C_c^\infty(B(\bar{x},6\delta))$. Let $\tau>0$ and $u,v$ two solutions of our problem such that $v\geq u$ on $B_C(\bar{x},6\delta)$.
Define, for any $\beta<0$, 
$$
 \phi_\tau:= \varphi^2w_\tau^\beta, \mbox{ where } w_\tau:=v-u+\tau.
$$
We begin by observing that, since we assume $u,v \in C_{loc}^{1,\alpha}$, the argument must be carried out on a subdomain $C_+'$. For notational convenience, we denote this subdomain by $C_+$ throughout the proof.
By a density argument, we can use $\phi_\tau$ as a test function in the weak formulation of \eqref{main_problem} and, by subtracting, it follows that
\begin{eqnarray*}
 &&\int_{C_+}2\varphi w_\tau^\beta (|\nabla u|^{p-2}\nabla u-|\nabla v|^{p-2}\nabla v)\cdot\nabla \varphi \ dx\\
 &&\qquad+\beta \int_{C_+}\varphi^2 w_\tau^{\beta-1} (|\nabla u|^{p-2}\nabla u-|\nabla v|^{p-2}\nabla v)\cdot\nabla w_\tau \ dx\\
 &&\leq\int_{C_+}(f(u)-f(v))\phi_\tau \, dx.
\end{eqnarray*}
Taking the modulus on both sides of the previous equation and by \cite[Equation (2-2)]{D} (see also \cite{LI}), we obtain
\begin{eqnarray*}
 &&|\beta|c_2  \int_{C_+}\varphi^2 w_\tau^{\beta-1} (|\nabla u|+|\nabla v|)^{p-2}|\nabla w_\tau|^2 \, dx\\
 &&\quad\leq \int_{C_+}|f(v)-f(u)|\phi_\tau \, dx+\left| \int_{C_+}2\varphi w_\tau^\beta (|\nabla v|^{p-2}\nabla v-|\nabla u|^{p-2}\nabla u)\cdot\nabla \varphi \ dx\right|.
 \end{eqnarray*}
 Applying \cite[Equation (2-1)]{D} (see also \cite{LI}) and then the Young inequality  
\begin{eqnarray*}
 &&|\beta|c_2  \int_{C_+}\varphi^2 w_\tau^{\beta-1} (|\nabla u|+|\nabla v|)^{p-2}|\nabla w_\tau|^2 \, dx\\
 &&\quad\leq \int_{C_+}|f(u)-f(v)|\phi_\tau \, dx+2\varepsilon c_1 \int_{C_+}\varphi^2w_\tau^{\beta-1} (|\nabla u|+|\nabla v|)^{p-2}|\nabla w_\tau|^2 \ dx \\
 &&\qquad\quad+\frac{2c_1}{\varepsilon} \int_{C_+}(|\nabla u|+|\nabla v|)^{p-2}|\nabla\varphi|^2 w_\tau^{\beta+1} \ dx. 
 \end{eqnarray*} 
 Let us choose $\varepsilon=\frac{c_2|\beta|}{2c_1}$ in such a way that we can get
  \begin{eqnarray*}
 &&\frac{|\beta|c_2}{2}  \int_{C_+}\varphi^2 w_\tau^{\beta-1} (|\nabla u|+|\nabla v|)^{p-2}|\nabla w_\tau|^2 \, dx\\
 &&\qquad\leq \int_{C_+}|f(u)-f(v)|\phi_\tau \, dx+\frac{4c_1^2}{c_2|\beta|} \int_{C_+}(|\nabla u|+|\nabla v|)^{p-2}|\nabla\varphi|^2 w_\tau^{\beta+1} \ dx. 
 \end{eqnarray*} 
 By the local Lipschitz continuity of $f$, there exists a positive constant $L$ depending on $\|v\|_\infty$ such that
 \begin{equation*}
 \begin{split}
   \int_{C_+} |f(u)-f(v)|\phi_\tau dx&\leq L \int_{C_+}|u-v|\phi_\tau \, dx\\
   &\leq  L \int_{C_+}|u-v+\tau|\phi_\tau \, dx\leq  L \int_{C_+}\varphi^2 w_{\tau}^{\beta+1}\ dx.
\end{split}
\end{equation*}
 Thus, choosing $\tilde{C}:=\tilde{C}(L, p, \|v\|_{L^{\infty}},  \|\nabla u\|_{L^{\infty}},   \|\nabla v\|_{L^{\infty}})>\frac{2}{c_2}\max\{L,2c_1\}$, we obtain
  \begin{eqnarray}\label{app1}
 \nonumber &&\int_{C_+}\varphi^2 w_\tau^{\beta-1} (|\nabla u|+|\nabla v|)^{p-2}|\nabla w_\tau|^2 \, dx\\
 \nonumber && \qquad\leq \frac{2L}{|\beta|c_2} \int_{C_+}\varphi^2 w^{\beta+1}_\tau \, dx +\frac{8c_1^2}{|\beta|^2c_2^2} \int_{C_+}(|\nabla u|+|\nabla v|)^{p-2}|\nabla\varphi|^2 w_\tau^{\beta+1} \ dx \\
 &&\qquad\leq \frac{\tilde{C}}{|\beta|}\left(1+\frac{\tilde{C}}{|\beta|} \right) \int_{C_+}w_\tau^{\beta+1}(\varphi^2 +(|\nabla u|+|\nabla v|)^{p-2}|\nabla\varphi|^2) \ dx.  
 \end{eqnarray}  
\ \\
\noindent Let us now consider $\beta\neq -1$ and
$$
 \tilde{w}_\tau=w_\tau^{\frac{\beta+1}{2}}.
$$
 Using \eqref{app1}, we note that
 \begin{eqnarray}\label{monty*}
 \nonumber &&\int_{C_+}\varphi^2 (|\nabla u|+|\nabla v|)^{p-2}|\nabla \tilde{w}_\tau|^2 \, dx=\frac{(\beta+1)^2}{4}\int_{C_+}\varphi^2 w_\tau^{\beta-1} (|\nabla u|+|\nabla v|)^{p-2}|\nabla w_\tau|^2\, dx \\
&&\qquad\quad\leq\frac{(\beta+1)^2}{4} \frac{\tilde{C}}{|\beta|}\left(1+\frac{\tilde{C}}{|\beta|} \right) \int_{C_+}\tilde{w}_\tau^2(\varphi^2 +(|\nabla u|+|\nabla v|)^{p-2}|\nabla\varphi|^2) \ dx.
 \end{eqnarray}
Without loss of generality, we can eventually consider $\delta>0$ small enough such that our ball $B_{C}(\bar{x},7\delta)$ is completely contained in a ball where the unique nearest point property holds. By Fermi coordinates and the flattening operator, we define $\mathcal{B}_C^+$ as $B_{C}(\bar{x},6\delta)=\Phi(\mathcal{B}_C^+)$; then we consider  $\mathcal{B}_C^-$ defined as the reflection of $\mathcal{B}_C^+$ with respect to the hyperplane $y_N=0$ and we extend the function $(\varphi \tilde{w}_\tau)(\Phi(y))$ by its even extension on  $\mathcal{B}:=\mathcal{B}_C^+\cup\mathcal{B}_C^-$.   
Denoting by $\tilde z(y)$ the above extension on $\mathcal{B}$ and by
$$
  d_J(y):=|{det} J_\Phi(y)|, \ A(y):=[J_\Phi(y)^T]^{-1},
$$
we have
 \begin{eqnarray*}
 \int_{\mathcal{B}}|\tilde{z}(y)|^\nu d_J(y)dy\leq (\sup_{\mathcal{B}}d_J(y))  \int_{\mathcal{B}}|\tilde{z}(y)|^\nu dy.
 \end{eqnarray*} 
Using the weighted Sobolev inequality, available in the case of the weight $(|\nabla_y \tilde{u}|+|\nabla_y \tilde{v}|)^{p-2}$, for $\nu\in (2,\bar{2}_p)$ (where $\tilde{u}:=u\circ\Phi$) (thanks to \cite[Proposition 6.1]{MMS}, we obtain
\begin{equation*}
\begin{split}
\left(\int_{\mathcal{B}}|\tilde{z}(y)|^\nu d_J(y)dy\right)^\frac{2}{\nu}&\leq\left(\sup_{\mathcal{B}}d_J(y)\right)^\frac{2}{\nu} \left(\int_{\mathcal{B}}|\tilde{z}(y)|^\nu dy\right)^\frac{2}{\nu}\\
&\leq \|d_J\|_\infty^\frac{2}{\nu} C_p \int_{\mathcal{B}}
|\nabla_y\tilde{z}(y)|^2(|\nabla_y \tilde{u}|+|\nabla_y \tilde{v}|)^{p-2}dy\\
 &\leq \frac{\|d_J\|_\infty^\frac{2}{\nu}\sup_{\mathcal {B}}\|J_\Phi^T\|^{p}C_p}{\inf_{\mathcal {B}}d_J}\\
 &\qquad\qquad\times\int_{\mathcal{B}}
|A(y)\nabla_y\tilde{z}(y)|^2(|A(y)\nabla_y \tilde{u}|+|A(y)\nabla_y \tilde{v}|)^{p-2}d_J(y)dy.
\end{split}
\end{equation*} 
Note that the constant
$$
\frac{\|d_J\|_\infty^\frac{2}{\nu}\sup_{\mathcal {B}}\|J_\Phi^T\|^{p}C_p}{\inf_{\mathcal {B}}d_J},
$$
is bounded because our flattening operator is, indeed, a diffeomorphism.\\
Then, coming back to our problem, we deduce that there exists a suitable constant $\bar{C}>0$ such that
\begin{eqnarray*}
 \|\varphi \tilde{w}_\tau\|^2_{L^\nu}&\leq& 2^{-\frac{2}{\nu}} \left( \int_{\mathcal{B}}|\tilde{z}(y)|^\nu\delta(y)dy\right)^{\frac{2}{\nu}}\\
&\leq& \bar{C} \int_{\mathcal{B}}|A(y)\nabla_y\tilde{z}(y)|^2(|A(y)\nabla_y \tilde{u}|+|A(y)\nabla_y \tilde{v}|)^{p-2}d_J(y)dy\\
&=& \bar{C}\int_{C_+}(|\nabla u|+|\nabla v|)^{p-2}|\nabla(\varphi\tilde{w}_\tau)|^2 \ dx\\
&\leq&2 \bar{C}\int_{C_+}(|\nabla u|+|\nabla v|)^{p-2}(\tilde{w}_\tau^2|\nabla\varphi|^2+\varphi^2|\nabla\tilde{w}_\tau|^2) \ dx\\
\mbox{ by \eqref{monty*}}&\leq&2 \bar{C}\int_{C_+}(|\nabla u|+|\nabla v|)^{p-2}\tilde{w}_\tau^2|\nabla\varphi|^2\ dx\\
&& + \bar{C}\frac{(\beta+1)^2}{2} \frac{\tilde{C}}{|\beta|}\left(1+\frac{\tilde{C}}{|\beta|} \right) \int_{C_+}\tilde{w}_\tau^{2}(\varphi^2 +(|\nabla u|+|\nabla v|)^{p-2}|\nabla\varphi|^2) \ dx \\
&=&\bar{C}\frac{(\beta+1)^2}{|\beta|}\left(1+\frac{\tilde{C}}{|\beta|} \right) \int_{C_+}\tilde{w}_\tau^{2}(\varphi^2 +(|\nabla u|+|\nabla v|)^{p-2}|\nabla\varphi|^2) \ dx,
\end{eqnarray*}
up to redefine in a suitable way the constant $\bar{C}$.\\
Referring to \cite[Equation (3.17)]{MeMoSc} and exploiting a similar computation up to \cite[Equation (3.26)]{MeMoSc}, we obtain that for all $r_0>0$, there exists a positive constant $C>0$ such that
$$
 \phi(-\infty,\delta,w_\tau)\geq C\phi\left(-r_0,\frac52\delta,w_\tau\right),
$$
that is Stp(i).

\noindent \emph{Proof of Stp(ii).}  Using $\beta=-1$ in \eqref{app1}, we have
\begin{eqnarray}\label{unavolta}
  &&\int_{C_+}\varphi^2 w_\tau^{-2}(|\nabla u|+|\nabla v|)^{p-2}|\nabla w_\tau|^2 \, dx
 \leq {\tilde{C}}(1+\tilde{C}) \int_{C^+}\varphi^2 +(|\nabla u|+|\nabla v|)^{p-2}|\nabla\varphi|^2\ dx.  
\end{eqnarray}
 Let us now consider $\tilde{w}_\tau:=\log(w_\tau)$; then we can rewrite \eqref{unavolta} as
 \begin{equation}\label{unavolta2}
  \int_{C_+}\varphi^2(|\nabla u|+|\nabla v|)^{p-2}|\nabla \tilde{w}_\tau|^2  \, dx \leq {\tilde{C}}(1+\tilde{C}) \int_{C^+}\varphi^2 +(|\nabla u|+|\nabla v|)^{p-2}|\nabla\varphi|^2\ dx.  
\end{equation}
Denoting by $\tilde{w}_\tau^\Phi$ and $\varphi_\Phi$ the trasformated functions by Fermi coordinate fo $\tilde{w}_\tau$ and $\varphi$ respectively, the previous becomes
\begin{eqnarray*}
  &&\int_{supp(\varphi_\Phi)}\varphi_\Phi^2(|A(y)\nabla \tilde{u}|+|A(y)\nabla \tilde{v}|)^{p-2}|A(y)\nabla \tilde{w}_\tau^\Phi|^2  \, d_J(y)dy\\ &&\qquad\leq {\tilde{C}}(1+\tilde{C}) \int_{supp(\varphi_\Phi)}\varphi_\Phi^2 +(|A(y)\nabla \tilde{u}|+|A(y)\nabla \tilde{v}|)^{p-2}|A(y)\nabla\varphi_\Phi|^2\ d_J(y) dy.  
 \end{eqnarray*}
Subsequently, we need to convexify the $\Phi^{-1}(B_C(\bar{x},5\delta))$; this is because, later on, we want to apply a Poincaré–Sobolev inequality for functions that do not have zero trace on the boundary of the domain where we are working. This requires the convexity of the integration domain (we refer to \cite{FMS}).
Note that, by the regularity of $\Phi$ and computing the convex envelope $co(\Phi^{-1}(B_C(\bar{x},5\delta)))\subset \Phi^{-1}(B_C(\bar{x},6\delta))$, choosing $\varphi_\Phi$ such that  $\varphi_\Phi=1$ in $co(\Phi^{-1}(B_C(\bar{x},5\delta)))$,  the inequality above implies that
\begin{eqnarray}\label{risol}
\nonumber && \int_{co(\Phi^{-1}(B_C(\bar{x},5\delta)))} (|A(y)\nabla \tilde{u}|+|A(y)\nabla \tilde{v}|)^{p-2}|A(y)\nabla \tilde{w}_\tau^\Phi|^2\, d_J(y) dy\\
\nonumber &&\qquad \leq\int_{supp(\varphi_\Phi)}\varphi_\Phi^2 (|A(y)\nabla \tilde{u}|+|A(y)\nabla \tilde{v}|)^{p-2}|A(y)\nabla \tilde{w}_\tau^\Phi|^2 \, d_J(y)dy\\ &&\qquad \leq {\tilde{C}}(1+\tilde{C}) \int_{supp(\varphi_\Phi)}\varphi_\Phi^2 +(|A(y)\nabla \tilde{u}|+|A(y)\nabla \tilde{v}|)^{p-2}|A(y)\nabla\varphi_\Phi|^2\ d_J(y) dy \leq \mathcal{C}, 
 \end{eqnarray}
where the constant $\mathcal{C}$ does not depend on $\tau$. As shown in \cite{MeMoSc, Mo}, we can assume that $\tilde{w}_\tau(=\log(w_\tau))$ has zero mean in $co(\Phi^{-1}(B_C(\bar{x},5\delta)))$. Indeed, for any $k>0$ we can rewrite \eqref{unavolta} to obtain the same \eqref{unavolta2} using $\log(w_\tau/k)$ instead of $\log(w_\tau)$. From \eqref{unavolta2} to \eqref{risol} this change does not introduce modification. Nevertheless, choosing 
$$
 k=\exp\left(|co(\Phi^{-1}(B_C(\bar{x},5\delta)))|^{-1}\int_{co(\Phi^{-1}(B_C(\bar{x},5\delta)))}log({w}^\Phi_\tau) dx\right)
$$
we easily get that $\log(w_\tau/k)$ has zero mean on $co(\Phi^{-1}(B_C(\bar{x},5\delta)))$ and this constant have not modified the computations \eqref{unavolta}-\eqref{risol}.\\
In this setting, we are allowed to expoit the weighted Poincar\'e-Sobolev type inequality; 
\begin{eqnarray*}
&& \int_{co(\Phi^{-1}(B_C(\bar{x},5\delta)))}| \tilde{w}_\tau^\Phi|^\nu \ dy\leq C_S \int_{co(\Phi^{-1}(B_C(\bar{x},5\delta)))} (|A(y)\nabla \tilde{u}|+|A(y)\nabla \tilde{v}|)^{p-2}|\nabla \tilde{w}_\tau^\Phi|^2\, dy\\
&&\qquad\leq C_S\|J_\Phi^T\|_\infty\int_{co(\Phi^{-1}(B_C(\bar{x},5\delta)))} (|A(y)\nabla \tilde{u}|+|A(y)\nabla \tilde{v}|)^{p-2}|A(y)\nabla \tilde{w}_\tau^\Phi|^2\, dy\\
&&\qquad\leq \frac{C_S\|J_\Phi^T\|_\infty}{\inf|d_J|}\int_{co(\Phi^{-1}(B_C(\bar{x},5\delta)))} (|A(y)\nabla \tilde{u}|+|A(y)\nabla \tilde{v}|)^{p-2}|A(y)\nabla \tilde{w}_\tau^\Phi|^2\,d_J(y) dy\\
&&\qquad\leq \mathcal{C},
\end{eqnarray*}
where the constant $\mathcal{C}$ does not depend on $\tau$.
Then, we are able to conclude that
$$
 \|\tilde{w}_\tau\|_{L^\nu(B_C(\bar{x},5\delta))}\leq \mathcal{C},
$$ 
and the constant does not depend on $\tau$.\\
Let us now define the test function
$$
 \phi_\tau:=\varphi^2\frac{(|\tilde{w}_\tau|^\beta+(2\beta)^\beta)}{w_\tau},
$$
where $\beta\geq 1$, $\varphi\in C_0^1(B(\bar{x},5\delta))$ and $\varphi$ is nonnegative. Using $\phi_\tau$ as test function in the weak formulations solved by $u$ and $v$ and then subtracting, we get
\begin{eqnarray}\label{a3}
   &&\int_{C_+} 2\varphi\frac{(|\tilde{w}_\tau|^\beta+(2\beta)^\beta)}{w_\tau} (|\nabla v|^{p-2}\nabla v-|\nabla u|^{p-2}\nabla u)\cdot\nabla \varphi \ dx\\
\nonumber &&\qquad+\int_{C_+} (|\nabla v|^{p-2}\nabla v-|\nabla u|^{p-2}\nabla u)\cdot\nabla w_\tau \ \frac{\varphi^2}{w_\tau^2}(\beta \ sign(\tilde{w}_\tau) |\tilde{w}_\tau|^{\beta-1}-|\tilde{w}_\tau|^\beta-(2\beta)^\beta)\, dx\\
\nonumber &&\quad\geq\int_{C_+} (f(v)-f(u))\varphi^2\frac{(|\tilde{w}_\tau|^\beta+(2\beta)^\beta)}{w_\tau} \, dx.  
\end{eqnarray}
Following \cite[Equations (3.42)-(3.44)]{MeMoSc}, we can estimate the second integral on the l.h.s. of \eqref{a3} as follows
\begin{eqnarray*}
 &&\int_{C_+} (|\nabla v|^{p-2}\nabla v-|\nabla u|^{p-2}\nabla u)\cdot\nabla w_\tau \ \frac{\varphi^2}{w_\tau^2}(\beta \ sign(\tilde{w}_\tau) |\tilde{w}_\tau|^{\beta-1}-|\tilde{w}_\tau|^\beta-(2\beta)^\beta) \, dx\\
 &&\qquad\leq-\beta \int_{C_+}(|\nabla v|+|\nabla u|)^{p-2}\varphi^2|\nabla \tilde{w}_\tau|^2|\tilde{w}_\tau|^{\beta-1} \, dx.
  \end{eqnarray*}
For the first integral on the l.h.s of \eqref{a3}, using \cite[Equation (2-1)]{D}, we have
\begin{eqnarray*}
 &&\int_{C_+} 2\varphi\frac{(|\tilde{w}_\tau|^\beta+(2\beta)^\beta)}{w_\tau} (|\nabla v|^{p-2}\nabla v-|\nabla u|^{p-2}\nabla u)\cdot\nabla \varphi \ dx\\
 &&\qquad\leq \tilde{C}\int_{C_+}(|\nabla u|+|\nabla v|)^{p-2} \ \frac{\varphi}{w_\tau}(|\tilde{w}_\tau|^\beta+(2\beta)^\beta)|\nabla \varphi||\nabla w_\tau| \, dx,
   \end{eqnarray*}
 for a suitable constant $\tilde{C}>0$. Using the Young inequality, we obtain 
   \begin{eqnarray*}
&& \tilde{C}\int_{C_+}(|\nabla u|+|\nabla v|)^{p-2} \ \frac{\varphi}{w_\tau}(|\tilde{w}_\tau|^\beta+(2\beta)^\beta)|\nabla \varphi||\nabla w_\tau| \, dx\\
&&\quad\leq \frac{\theta\tilde{C}}{2}\int_{C_+}(|\nabla u|+|\nabla v|)^{p-2}\varphi^2|\nabla \tilde{w}_\tau|^2|\tilde{w}_\tau|^{\beta-1}\, dx+\frac{\tilde{C}}{2\theta} \int_{C_+}(|\nabla u|+|\nabla v|)^{p-2}|\nabla \varphi|^2|\tilde{w}_\tau|^{\beta+1} \, dx\\
&&\quad\qquad+\frac{\tilde{C}}{2}\int_{C_+}(|\nabla u|+|\nabla v|)^{p-2}|\nabla \varphi|^2|(2\beta)^{\beta}\, dx+\frac{\tilde{C}}{2}\int_{C_+}(|\nabla u|+|\nabla v|)^{p-2}\varphi^2|\nabla \tilde{w}_\tau|^2(2\beta)^{\beta} \, dx. 
\end{eqnarray*} 
Taking into account the above inequality and \eqref{a3}, we have
\begin{eqnarray*}
&&\int_{C_+} (f(v)-f(u))\varphi^2\frac{(|\tilde{w}_\tau|^\beta+(2\beta)^\beta)}{w_\tau} \, dx\leq -\beta \int_{C_+}(|\nabla v|+|\nabla u|)^{p-2}\varphi^2|\nabla \tilde{w}_\tau|^2|\tilde{w}_\tau|^{\beta-1}\, dx\\
\nonumber &&\quad+\frac{\tilde C \theta}{2}\int_{C_+}(|\nabla u|+|\nabla v|)^{p-2}\varphi^2|\nabla \tilde{w}_\tau|^2|\tilde{w}_\tau|^{\beta-1}\, dx+\frac{\tilde C}{2\theta}\int_{C_+}(|\nabla u|+|\nabla v|)^{p-2}|\nabla \varphi|^2|\tilde{w}_\tau|^{\beta+1}\, dx\\
\nonumber &&\quad+\frac{\tilde{C}}{2}\int_{C_+}(|\nabla u|+|\nabla v|)^{p-2}|\nabla \varphi|^2|(2\beta)^{\beta}\, dx+\frac{\tilde{C}}{2}\int_{C_+}(|\nabla u|+|\nabla v|)^{p-2}\varphi^2|\nabla \tilde{w}_\tau|^2(2\beta)^{\beta}\, dx, 
\end{eqnarray*}
and then
\begin{eqnarray*}
&&\beta \int_{C_+}(|\nabla v|+|\nabla u|)^{p-2}\varphi^2|\nabla \tilde{w}_\tau|^2|\tilde{w}_\tau|^{\beta-1}\, dx\leq \int_{C_+} |f(v)-f(u)|\varphi^2\frac{(|\tilde{w}_\tau|^\beta+(2\beta)^\beta)}{w_\tau}\, dx\\
\nonumber &&\quad+\frac{\theta\tilde{C}}{2}\int_{C_+}(|\nabla u|+|\nabla v|)^{p-2}\varphi^2|\nabla \tilde{w}_\tau|^2|\tilde{w}_\tau|^{\beta-1}\, dx+\frac{\tilde{C}}{2\theta}\int_{C_+}(|\nabla u|+|\nabla v|)^{p-2}|\nabla \varphi|^2|\tilde{w}_\tau|^{\beta+1}\, dx\\
\nonumber &&\quad+\frac{\tilde{C}}{2}\int_{C_+}(|\nabla u|+|\nabla v|)^{p-2}|\nabla \varphi|^2|(2\beta)^{\beta}\, dx+\frac{\tilde{C}}{2}\int_{C_+}(|\nabla u|+|\nabla v|)^{p-2}\varphi^2|\nabla \tilde{w}_\tau|^2(2\beta)^{\beta}\, dx.
\end{eqnarray*}
Regarding the integral involving the source term $f$, by the local Lipschitz continuity, we have
\begin{eqnarray*}
  \int_{C_+} |f(v)-f(u)|\varphi^2\frac{(|\tilde{w}_\tau|^\beta+(2\beta)^\beta)}{w_\tau}\, dx \leq L  \int_{C_+} \varphi^2(|\tilde{w}_\tau|^\beta+(2\beta)^\beta)\, dx. 
   \end{eqnarray*}  
 Collecting the above inequalities, taking $\theta$ small enough, we get
 \begin{eqnarray*}
  &&\beta \int_{C_+}(|\nabla u|+|\nabla v|)^{p-2}\varphi^2|\nabla \tilde{w}_\tau|^2|\tilde{w}_\tau|^{\beta-1}\, dx\leq L  \int_{C_+} \varphi^2(|\tilde{w}_\tau|^\beta+(2\beta)^\beta) \, dx\\
  &&\qquad+ \tilde{C}\int_{C_+}(|\nabla u|+|\nabla v|)^{p-2}|\nabla \varphi|^2|((2\beta)^{\beta}+|\tilde{w}_\tau|^{\beta+1})\, dx\\
  &&\qquad+\frac{\tilde{C}}{2}\int_{C_+}(|\nabla u|+|\nabla v|)^{p-2}\varphi^2|\nabla \tilde{w}_\tau|^2(2\beta)^{\beta}\, dx.
   \end{eqnarray*}
Recalling that $\varphi\in C_0^1(B(\bar{x},5\delta))$ and \eqref{risol} holds, we obtain the following estimate
$$
 \frac{\tilde{C}}{2}\int_{C_+}(|\nabla u|+|\nabla v|)^{p-2}\varphi^2|\nabla \tilde{w}_\tau|^2(2\beta)^{\beta}\, dx \leq (2\beta)^\beta \tilde{C}.
 $$
 Therefore, we deduce
 \begin{eqnarray*}
  &&\beta \int_{C_+}(|\nabla u|+|\nabla v|)^{p-2}\varphi^2|\nabla \tilde{w}_\tau|^2|\tilde{w}_\tau|^{\beta-1}\, dx\leq L  \int_{C_+} \varphi^2(|\tilde{w}_\tau|^\beta+(2\beta)^\beta)\, dx \\
  &&\qquad+ \tilde{C}\int_{C_+}(|\nabla u|+|\nabla v|)^{p-2}|\nabla \varphi|^2((2\beta)^{\beta}+|\tilde{w}_\tau|^{\beta+1})\, dx\\
  &&\qquad +(2\beta)^\beta \tilde{C}. 
   \end{eqnarray*} 
Exploiting the fact that the support of $\varphi$ depends on $\delta$, we have
\begin{equation*}
    (2\beta)^\beta = C(\delta) \int_{B(\bar x,5\delta)} \varphi^2(2\beta)^\beta \, dx.
\end{equation*}
Thus, we can rewrite the previous inequality as follows
 \begin{eqnarray*}
  &&\beta \int_{C_+}(|\nabla u|+|\nabla v|)^{p-2}\varphi^2|\nabla \tilde{w}_\tau|^2|\tilde{w}_\tau|^{\beta-1}\, dx\leq \bar{C}  \int_{C_+} \varphi^2(|\tilde{w}_\tau|^\beta+(2\beta)^\beta) \, dx\\
  &&\qquad+ \tilde{C}\int_{C_+}(|\nabla u|+|\nabla v|)^{p-2}|\nabla \varphi|^2((2\beta)^{\beta}+|\tilde{w}_\tau|^{\beta+1})\, dx,
   \end{eqnarray*} 
where $\bar{C}=\bar C(L,\beta,\delta)$ is a positive constant. Using \cite[Equation (3.42)]{MeMoSc}, we deduce that
$$
 \int_{C_+} \varphi^2(|\tilde{w}_\tau|^\beta+(2\beta)^\beta)\, dx \leq  c(\beta)\int_{C_+} \varphi^2(|\tilde{w}_\tau|^{\beta+1}+(2\beta)^\beta)\, dx.  
 $$
 Hence, we conclude that
 \begin{eqnarray*}
  &&\beta \int_{C_+}(|\nabla u|+|\nabla v|)^{p-2}\varphi^2|\nabla \tilde{w}_\tau|^2|\tilde{w}_\tau|^{\beta-1}\, dx\\
  &&\qquad\leq  \tilde{C}\int_{C_+}(\varphi^2+(|\nabla u|+|\nabla v|)^{p-2}|\nabla \varphi|^2)((2\beta)^{\beta}+|\tilde{w}_\tau|^{\beta+1})\, dx.
   \end{eqnarray*} 
 Let us remark that the previous inequality is very similar to \eqref{app1}, except for the term $(2\beta)^\beta$. Hence, in a similar way we can prove that
 $$
 \phi\left(r_0,\frac52\delta,w_\tau\right)\leq C \phi\left(-r_0,\frac52\delta,w_\tau\right),
 $$ 
 that is Stp(ii).\\ 
 The previous estimate, together with
 $$
 \phi(-\infty,\delta,w_\tau)\geq C\phi\left(-r_0,\frac52\delta,w_\tau\right),
$$
proved in Stp(i), allows us to conclude that
$$
 \phi(-\infty,\delta,w_\tau)\geq C \phi\left(r_0,\frac52\delta,w_\tau\right).
 $$  
Fix $s\in(0,r_0]$, the proof concludes by \cite[Equations (3.28)-(3.29)]{MeMoSc}.
If instead, $s\in (r_0, \frac{\bar{ 2}_p}{2})$, then we get the thesis by \cite[Equations (3.31)-(3.35)]{MeMoSc}.
\end{proof}
 
The consequence of the weak Harnack comparison inequality, previously proved, is the following strong comparison principle on the lateral surface of the cylinder.
 
 \begin{theorem}\label{thm:strongggggggg}
Let $p>2$, assume that
$f(u)$ fulfills $(H)$ and $N\geq 2$.  Let $u\in C^{1,\alpha}_{loc}(\overline{C}_+)$ be a weak solution of \eqref{main_problem}. Let $v\in C^{1,\alpha}_{loc}(\overline{C}_\sigma)$ satisfying
\begin{equation*}
    - \Delta_p u - f(u) \leq - \Delta_p v - f(v), \quad u \leq v \quad \text{in} \ C_\sigma, \mbox{ for some }\sigma>0,
\end{equation*}
in the weak distributional meaning. Then for any connected component $\mathcal{C}_i$ of $\partial\Omega\times (0,\sigma)$ 
it follows that
$$u<v \quad \text{in} \quad\mathcal{C}_i,$$
unless $u\equiv v$ in $\mathcal{C}_i$.
\end{theorem}
\begin{proof}
 Denote $\mathcal{N}$ as the (closed) subset of $\partial \Omega\times (0,\sigma)$ where $u=v$. Denote by $C_{\mathcal{N}}^i:=\mathcal{C}_i\cap\mathcal{N}$ that is closed, since it is the intersection of closed sets. Moreover, for each $x\in C_{\mathcal{N}}^i$, by Theorem \ref{WHarna}, there exists a suitable ball $B_C(x,\delta)\subset C_\sigma$ where $u=v$ (by continuity). Then $\overline{B_C(x,\delta)}\cap \mathcal{C}_i\subset \mathcal{N}$ and so $x$ is an interior point w.r.t the topology of the $\partial \Omega\times (0,\sigma)$. Hence, $C_{\mathcal{N}}^i$ is open in $\mathcal{C}_i$. Since $\mathcal{C}_i$ is connected, then either $u=v$ or $u<v$ on $\mathcal{C}_i$.
\end{proof}
\section{Proof of the Monotonicity Result - Theorem \ref{Monotonicity}}\label{monomono}
We first state and prove the following proposition, which provides a preliminary step in the proof of our monotonicity result. Indeed, this proposition provides a useful Hopf Lemma-type result for our problem \eqref{main_problem}. 
\begin{proposition}\label{N_mon}
Let $u \in C^{1,\alpha}_{loc}(\overline{C}_+)$ be a weak solution of \ref{main_problem}. Then, there exists $\tilde{x}_N>0$ such that $\partial_{x_N} u > 0$ in $\overline{C}_{\tilde{x}_N}$.
\end{proposition}
\begin{proof}
Notice that, since $u=0$ on $\Omega \times \{0\}$, it follows via standard Hopf Lemma that $\partial_{x_N} u(x',0) > 0$, for any $x' \in \Omega$.\ \\
If $x' \in \partial \Omega$, the proof is more involved and requires two steps.
We divide the proof accordingly.\\
We can prove, though the argument is purely technical and therefore omitted, that there exists, for some $\hat t > 0$ small, a positive solution $w$, in $(0,\hat t)$, of the one dimensional problem
\begin{equation}\label{1D}
\begin{cases}
(|w'|^{p-2}w')' = k w^{p-1} & \text{in } (0,\hat t),\\
w(0)=0, \quad \eta_0=w'(0)>0,
\end{cases}
\end{equation}
where $k$ is a positive constant.
The second step consists in comparing the solution $w$ of the one-dimensional problem with the solution $u$ of \eqref{main_problem} in $C_{\hat{x}_N}$ for a suitable $\hat x_N>0$. \ \\
In what follows, we will consider $w$ as a function defined in $\mathbb{R}^N$ depending only on the $x_N$ variable, namely
\begin{equation*}
w : \mathbb{R}^N \rightarrow \mathbb{R}^+: \ w(x_1,x_2,...,x_N) = w(x_N).
\end{equation*}
Note that the boundary of $C_{\hat{x}_N}$ can be divided in three parts, namely:
\begin{equation*}
\partial C_{\hat{x}_N} = \overline{\Omega}_0 \cup \overline{\Omega}_{\hat{x}_N} \cup A_1,
\end{equation*}
where,
\begin{equation*}
 \Omega_0=\Omega\times\{0\},\qquad \Omega_{\hat{x}_N} := \Omega \times \{\hat{x}_N\}, \quad A_1 := \partial \Omega \times (0,\hat{x}_N).
\end{equation*}
In $\overline{\Omega}_0$,  by the Dirichlet condition on \eqref{main_problem},  we have $u=0$ and $w=0$.\ \\
In $A_1$,  recalling \eqref{main_problem},  we have $\partial_{\nu} u =0$ and $\partial_{\nu} w = 0$,   since $w$ depends only on $x_N$. \ \\
In $\overline{\Omega}_{\hat{x}_N}$,  since the function $\tilde w$, defined as follows:
\begin{equation*}
\tilde w : \mathbb{R}^N \rightarrow \mathbb{R}^+, \ \ \tilde w :=\varepsilon w \ \text{for some } \varepsilon > 0,
\end{equation*}
is still a solution to \eqref{1D},  we can choose $\varepsilon>0$ sufficiently small so that $\tilde w < u$ in $\overline{\Omega}_{\hat{x}_N}$.\ \\
Thus, by a density argument, we can use $(\tilde w-u)^+$ as test function,  both in \eqref{main_problem} and \eqref{1D}, integrating over $C_{\hat{x}_N}$ .\ \\
For simplicity,  we shall continue to denote by $w$ the function $\tilde{w}$. Testing problem \eqref{1D} with $(w - u)^+$, we obtain
\begin{equation}\label{weak_1D}
\int_{[w \geq u]} |\nabla w|^{p-2} \nabla w \cdot \nabla (w-u)\, dx + k \int_{[w \geq u]} w^{p-1}(w-u)\, dx=0.
\end{equation}
On the other hand, recalling the positivity of $f$, integrating by parts \eqref{main_problem} with $(w-u)^+$, we get
\begin{equation}\label{weak_main}
\int_{[w \geq u]} |\nabla u|^{p-2} \nabla u \cdot \nabla(w-u) \ dx +k \int_{[w \geq u]} u^{p-1} (w-u) \ dx \geq 0.
\end{equation}
Subtracting \eqref{weak_main} from \eqref{weak_1D}, we have
\begin{equation*}
\int_{[w \geq u]} (|\nabla w|^{p-2} \nabla w - |\nabla u|^{p-2} \nabla u) \cdot \nabla (w-u) \ dx + k \int_{[w \geq u]}( w^{p-1}- u^{p-1} )(w-u) \ dx \leq 0.
\end{equation*}
Since the function $s^{p-1}$ is non-decreasing for $s \geq 0$,  the last integral is non-negative,  thus it follows that
\begin{equation*}
\int_{[w \geq u]} (|\nabla w|^{p-2} \nabla w - |\nabla u|^{p-2} \nabla u) \cdot \nabla (w-u) \ dx \leq 0.
\end{equation*}
By a standard inequality, see \cite{D,LI}, we obtain
\begin{equation*}
\int_{C_{\hat{x}_N}} (|\nabla w|+|\nabla u|)^{p-2} |\nabla(w-u)^+|^2 \ dx \leq 0.
\end{equation*}
This yields a contradiction unless $\nabla (w-u)^+ = 0$.  Since $w=u=0$ on $\Omega_0$ we conclude that 
\begin{equation*}
 w \leq u \quad \text{in } C_{\hat{x}_N}.
\end{equation*}
The continuity of $u$ and $w$ also implies that $w\leq u$ on $\overline{C}_{\hat{x}_N}$.
Therefore, the assumption on the problem \eqref{1D}, $w'(0)>0$, then implies $\partial_{x_N} u(x) \geq w_{x_N}(0)= w'(0) > 0$, on $\partial \Omega\times [0,\hat{x}_N)$. Thus, by continuity, there exists $\tilde{x}_N > 0$ small enough such that $\partial_{x_N} u >0$ in $\overline C_{\tilde{x}_N }$.
\end{proof}\ \\
\noindent We denote by $u^\theta$ the reflection of $u$, w.r.t. the hyperplane $\{x_N = \theta \}$, defined by 
$$
 u^\theta(x',x_N):=u(x',2\theta-x_N).
$$
\begin{proof}[Proof of Theorem \ref{Monotonicity}]
As a consequence of Proposition \ref{N_mon}, there exists a suitable $\theta'>0$ such that we have $u < u^{\theta'}$ in $ \overline \Omega \times [0,\theta')$.
Let us now define:
\begin{equation*}
\bar \lambda := \sup\{\theta' > 0 : u \leq u^\theta \ \text{in } C_{\theta}, \ \forall \theta \in (0,\theta')\}.
\end{equation*}
If $\bar \lambda=+ \infty$ we get the thesis; therefore let us suppose, by contraddiction, that $\bar \lambda<+ \infty$. This would imply that 
\begin{equation*}
u \leq u^{\bar \lambda} \quad \text{in } \ C_{\bar \lambda}=  \Omega \times (0, \bar \lambda].
\end{equation*}
\begin{remark}\label{rem1}
For the reader’s convenience, we would like to emphasize that the cases $1<p\leq2$ and $p>2$ are fundamentally different. In particular, in the former case, the moving plane method relies on a Poincaré inequality for functions with homogeneous Dirichlet data on a subset of the boundary of the domain under consideration. The latter case presents an additional difficulty since, as will become clear from the proof, the Poincaré inequality must be replaced by a suitable weighted version. To the best of our knowledge, no Poincaré inequalities, either standard or weighted, currently available in the literature are adequate for our purposes.
\end{remark}

\subsection{The case \texorpdfstring{$p \leq 2$}{p < = 2}}
\
We prove here the monotonicity result in the case $p\leq2$. 
We shall consider  separately the cases $\frac{2N+2}{N+2} < p \leq2$ and $1<p\leq\frac{2N+2}{N+2}$. In the case $\frac{2N+2}{N+2} < p \leq2$ we may benefit of a strong comparison principle so that the proof is easier. On the contrary some extra work is necessary to deal with the case $1<p\leq\frac{2N+2}{N+2}$. Although it would be possible to provide a proof that consider directly the two cases jointly, we prefer to keep both the two different approaches since these are interesting and of possible use in future developments.

\subsubsection{The case \texorpdfstring{$\frac{2N+2}{N+2} < p \leq2$}{2N+2/N+2 < p <= 2}}\hfill\\

By the strong comparison principle, see \cite[Theorem $3.7$]{D-S},  since $u=0$ and $u^{\bar \lambda} > 0$ in $\Omega_0$, we deduce that:
\begin{equation}\label{strict}
u < u^{\bar \lambda} \quad \text{in } \  C_{\bar \lambda}.
\end{equation}
Let us now consider a compact subset of $C_{\bar \lambda}$, namely $K := \Omega' \times [t_1,t_2]$, with $t_1 > 0$ small; $t_2 < \bar \lambda$ such that the difference $(\bar \lambda + \varepsilon) - t_2$ is small enough, and $\Omega' \subset \subset \Omega$ big enough. \ \\
We claim that $u < u^{\bar \lambda +\varepsilon}$,  in $K$ for some $\varepsilon >0$ small. Indeed, since $u^{\bar \lambda} - u \geq \delta > 0$ in $K$ by \eqref{strict}, by the uniform continuity, we infer that we can choose $\varepsilon>0$ such that:
\begin{equation}\label{Cont_comp}
u^{\bar \lambda + \varepsilon} - u = u^{\bar \lambda + \varepsilon} - u^{\bar \lambda} + u^{\bar \lambda} - u \geq \frac{\delta}{2} > 0 \quad \text{in } \ K.
\end{equation}
The remaining part consists in studying the solution $u$ in 
\begin{equation*}
D^*:=C_{\bar \lambda+\varepsilon} \setminus K.
\end{equation*}
The boundary of $D^*$ can be written as follows:
\begin{equation*}
\partial D^* = \overline \Omega_0\cup\overline \Omega_{\bar \lambda +\varepsilon} \cup \partial K \cup ( \partial \Omega \times (0,\bar \lambda + \varepsilon)).
\end{equation*}
By \eqref{Cont_comp}, $u \leq u^{\bar \lambda + \varepsilon}$ on $\partial K$.\\ 
In $\overline \Omega_{\bar \lambda + \varepsilon}$, by the definition of $u^{\bar \lambda + \varepsilon}$, it follows that $u=u^{\bar \lambda + \varepsilon}$.\\
In $\overline \Omega_0$ recall that $0=u<u^{\bar\lambda +\varepsilon}$ since for every $x\in\overline\Omega_0$,  $u^{\bar\lambda +\varepsilon}(x)>0$.\\
On $\partial \Omega \times (0,\bar \lambda + \varepsilon)$ we have the Neumann homogeneous boundary conditions, namely $\partial_{\nu}u = \partial_{\nu} u^{\bar\lambda + \varepsilon}=0$.\\
Our goal consists in deducing a weak comparison principle in $D^*$. To this aim, we note that, the reflected function $u^{\bar \lambda + \varepsilon}$ solves:
\begin{equation*}
\begin{cases}
- \Delta_p u^{\bar \lambda + \varepsilon} = f(u^{\bar\lambda + \varepsilon}) & \text{in } C_{\bar \lambda+\varepsilon},\\
\partial_{\nu} u^{\bar \lambda + \varepsilon} = 0 & \text{on } \partial \Omega \times (0,\bar \lambda + \varepsilon),
\end{cases}
\end{equation*}
and $u \leq u^{\bar \lambda + \varepsilon}$ on $\overline \Omega_0 \cup \overline \Omega_{\bar \lambda + \varepsilon} \cup \partial K$. 
Then the function $(u-u^{\bar \lambda+ \varepsilon})^+ \in W^{1,p}_D(C_{\bar \lambda+\varepsilon})$, where $D=K\cup \overline{\Omega}_0 \cup \overline{\Omega}_{\bar \lambda + \varepsilon}$, (see \eqref{test}). \\
Hence, we have
\begin{equation*}
\int_{D^*} |\nabla u|^{p-2} \nabla u \cdot \nabla (u - u^{\bar \lambda + \varepsilon})^+ \ dx = \int_{D^*} f(u) (u- u^{\bar \lambda + \varepsilon})^+ \ dx,
\end{equation*}
and
\begin{equation*}
\int_{D^*} |\nabla u^{\bar \lambda+ \varepsilon}|^{p-2} \nabla u^{\bar\lambda + \varepsilon} \cdot \nabla (u - u^{\bar\lambda + \varepsilon})^+ \ dx = \int_{D^*}f(u^{\bar \lambda+ \varepsilon}) (u- u^{\bar \lambda + \varepsilon})^+ \ dx.
\end{equation*}
Subtracting the two equations, 
\begin{equation*}
\begin{split}
\int_{D^*} (|\nabla u|^{p-2} \nabla u - |\nabla u^{\bar \lambda + \varepsilon}|^{p-2} \nabla u^{\bar \lambda + \varepsilon}) & \cdot \nabla (u - u^{\bar \lambda + \varepsilon})^+ \ dx\\
&= \int_{D^*} (f(u) - f(u^{\bar \lambda + \varepsilon}))(u- u^{\bar \lambda+ \varepsilon})^+ \ dx.
\end{split}
\end{equation*}
In addition, since $\nabla u$ and $\nabla u^{\bar \lambda + \varepsilon}$ are bounded in $D$, we have
\begin{equation}\label{Classic_I}
\begin{split}
\int_{D^*} (|\nabla u| + |\nabla u^{\bar \lambda + \varepsilon}|)^{p-2} & |\nabla (u - u^{\bar \lambda + \varepsilon})^+|^2 \ dx\\
&\geq C(p,\| \nabla u \|_{L^\infty(D)},\| \nabla u^{\bar \lambda + \varepsilon} \|_{L^\infty(D)}) \int_{D^*}  |\nabla (u - u^{\bar \lambda + \varepsilon})^+|^2 \ dx.
\end{split}
\end{equation}
By the Lipschitz continuity of $f$ and by \eqref{Classic_I}, it follows that
\begin{equation}\label{Pre_F}
\int_D |\nabla (u - u^{\bar \lambda + \varepsilon})^+|^2 \ dx\\
\leq C \int_D  |(u - u^{\bar \lambda + \varepsilon})^+|^2 \ dx,
\end{equation}
where $C=C(p,L,\|u\|_{\infty},\|\nabla u\|_\infty,\|\nabla u^{\bar \lambda + \varepsilon}\|_\infty)$ is a postive constant.\\
For the sake of simplicity, see the Figure \ref{fig:rettangolo-lambda}, let us consider the set $D^*$ as $D^*=D_1\cup D_2\cup D_3$, where
\begin{equation*}
    D_1 := \Omega \times (0,t_1), \quad D_2 := \Omega \times [t_1,t_2] \setminus K, \quad D_3 := \Omega \times (t_2,\bar \lambda + \varepsilon),
\end{equation*}
so that,
\begin{equation}\label{3pezzi}
\begin{split}
 \int_{D^*}  &|(u - u^{\bar \lambda + \varepsilon})^+|^2 \ dx \\
 &\qquad=  \int_{D_1}  |(u - u^{\bar \lambda + \varepsilon})^+|^2 \ dx +   \int_{D_2}  |(u - u^{\bar \lambda + \varepsilon})^+|^2 \ dx +  \int_{D_3}  |(u - u^{\bar \lambda + \varepsilon})^+|^2 \ dx\\
 &\qquad=: I_1+I_2+I_3.
\end{split}
\end{equation}
\begin{figure}[ht]
\centering
\begin{tikzpicture}[scale=1.2]

\draw[thick] (0,0) rectangle (3,4);
\draw[thick] (0.5,0.4) rectangle (2.5,3.2);

\draw[thick] (0,0) -- (3,0);        
\draw[thick] (0,4) -- (3,4);
\draw[thick] (0,3.6) -- (3,3.6); 

\node[left] at (0,0) {$0$};
\node[left] at (0,4) {$\bar{\lambda}+\varepsilon$};
\node[left] at (0,3.6) {$\bar{\lambda}$};
\node[left] at (0,3.2) {$t_2$};
\node[left] at (0,0.4) {$t_1$};
\node at (1.5,1.8) {$K$};
\node at (3.5,4.2) {$C_{\bar \lambda + \varepsilon}$};
\node at (1.5,3.5) {$D_3$};
\node at (1.5,0.2) {$D_1$};
\node at (0.25,1.8) {$D_2$};
\node at (2.75,1.8) {$D_2$};
\end{tikzpicture}
\caption{}
\label{fig:rettangolo-lambda}
\end{figure}\ \\
\ \\
First of all, we deal with $I_2$. Recalling the definition of $K$, then $D_2= (\Omega\setminus \Omega') \times [t_1,t_2]$, where $\Omega' \subset \subset \Omega$.
Taking into account Section 2, where we introduced the flattening operator, for a fixed $\bar x\in \partial \Omega\times[t_1,t_2]$, and for $r,\delta > 0$, we set
\begin{equation}\label{pallette}
    B_{\delta}^+ : =\{x\in \Omega \times [t_1,t_2]: d(x,\partial \Omega \times [t_1,t_2]) < \delta\} \cap B_r(\bar x).
\end{equation}
Let $\Phi : \mathcal{B}_{ft}^+ \rightarrow B_{\delta}^+$ the flattening operator. 
Let $\delta_0 > 0$ be such that $B_{\delta_0}^+$ has the unique nearest point property. Let $\Omega'$ be chosen sufficiently large so that, denoting by $\bar\delta := dist(\Omega,\Omega')$, we have $\bar\delta \leq \delta_0$.

Note that the compactness of $\partial \Omega \times [t_1,t_2]$ allows us to choose $\delta_0$ uniformly with respect to $\bar x \in \partial \Omega \times [t_1,t_2]$.

For clarity let $w(x):= (u(x) - u^{\bar \lambda + \varepsilon}(x))^+$, for all $x \in \Omega \times [t_1,t_2]$. Using the flattening operator previously defined, we have
\begin{equation}\label{eq_FLAT}
    \int_{B_{\bar \delta}^+} |w(x)|^2 \ dx = \int_{\mathcal{B}_{ft}^+} |w(\Phi(y))|^2 \cdot |\operatorname{det} J_\Phi(y)|\ dy,
\end{equation}
where $J_\Phi$ is the Jacobian matrix of $\Phi$, namely
\begin{equation}\label{J_PHI}
    J_\phi = \left[\partial_{y_1} \gamma + y_N \partial_{y_1} \eta,...,\partial_{y_{N-1}} \gamma+  y_N \partial_{y_{N-1}} \eta, \eta\right].
\end{equation}
For the reader’s convenience, we now provide an estimate of the determinant of the change of variables defined by Fermi coordinates; this type of estimate will be used several times in the text without being made explicit or directly referenced.
Note that $\partial_{y_i} \eta$ can be rewritten in terms of the second fundamental form associated to the parametrization $\bar x = \gamma (y_1,...,y_{N-1})$:
\begin{equation*}
    \partial_{y_i} \eta = - \sum_{j=1}^{N-1} \mathcal{B}_{i}^j \partial_{y_j} \gamma,
\end{equation*}
where $\mathcal{B}_{i}^j = \sum_{k=1}^{N-1} g^{jk} \mathcal{B}_{ki}$ represents the Weingarten map associated to the second fundamental form $\mathcal{B}_{ik}$ and $g^{jk}$ stands for the inverse matrix of $g_{jk}=\partial_{y_j}\gamma \cdot \partial_{y_k}\gamma$. Therefore, we have the following identity:
\begin{equation*}
    \partial_{y_i} \Phi = \sum_{j=1}^{N-1}(\delta_{ij} - y_N \mathcal{B}_{i}^j) \partial_{y_j} \gamma, \quad \text{for all } i=1,...,N-1.
\end{equation*}
Thus, \eqref{J_PHI} becomes
\begin{equation*}
\begin{split}
    J_\phi &= \left[\sum_j (\delta_{1j} - y_N\mathcal{B}_{1}^j)\partial_{y_j} \gamma,...,\sum_j (\delta_{N-1,j} - y_N\mathcal{B}_{N-1}^j)\partial_{y_j} \gamma , \eta\right]\\
    &=\left[\sum_j C_{j1}\partial_{y_j} \gamma,...,\sum_j C_{j,N-1}\partial_{y_j} \gamma , \eta\right],
\end{split}
\end{equation*}
where $C:= I - y_N \mathcal{B}$. Using the definition of the determinant and its multilinearity, it can be proved that
\begin{equation}\label{det}
    \operatorname{det} J_\Phi = \operatorname{det}\left[\sum_j C_{j1}\partial_{y_j} \gamma,...,\sum_j C_{j,N-1}\partial_{y_j} \gamma , \eta\right] = \operatorname{det}\left[\partial_{y_1} \gamma ,...,\partial_{y_{N-1}} \gamma, \eta\right] \operatorname{det} C.
\end{equation}
In addition, we know that
\begin{equation*}
    \operatorname{det} C = \operatorname{det}(I - y_N \mathcal{B}) = \prod_{j=1}^{N-1} (1-y_N k_j(y')),
\end{equation*}
where $k_1(y'),...,k_{N-1}(y')$ are the principal curvatures of the hypersurface parametrized by $\gamma$. Since $\partial \Omega \times [t_1,t_2]$ is smooth and compact in $\mathbb{R}^N$, it follows that $|k_j(y')|\leq K$, for any $j=1,...,N-1$. Moreover, we know that $y_N \in (0,\bar \delta)$ and we can choose $\bar \delta$ small such that $y_N < \frac{1}{2K}$. Thus, it easily follows that
\begin{equation*}
    \frac{1}{2} \leq |1-y_N k_j(y')| \leq \frac{3}{2}, \quad \text{ for any } j=1,...,N-1.
\end{equation*}
This would imply the following inequalities:
\begin{equation}\label{det_INEQ}
   \left(\frac{1}{2} \right)^{N-1}\leq |\operatorname{det}(I - y_N \mathcal{B})|\leq \left(\frac{3}{2}\right)^{N-1}.
\end{equation}
Using \eqref{det} and \eqref{det_INEQ}, \eqref{eq_FLAT} becomes
\begin{equation*}
\begin{split}
     \int_{B_{\bar \delta}^+} |w(x)|^2 \ dx &= \int_{\mathcal{B}_{ft}^+} |w(\Phi(y))|^2 \cdot \left|\operatorname{det}\left[\partial_{y_1} \gamma ,...,\partial_{y_{N-1}} \gamma, \eta\right]\right||\operatorname{det}(I - y_N \mathcal{B}) |\ dy\\
     &\leq \left(\frac{3}{2}\right)^{N-1} \int_{\mathcal{B}_{ft}^+} |w(\Phi(y))|^2 \cdot \left|\operatorname{det}\left[\partial_{y_1} \gamma ,...,\partial_{y_{N-1}} \gamma, \eta\right]\right| \ dy\\
     &=\left(\frac{3}{2}\right)^{N-1} \int_{A} \left(\int_{0}^{\delta} |w(\Phi(y))|^2 \ dy_N\right) \left|\operatorname{det}\left[\partial_{y_1} \gamma ,...,\partial_{y_{N-1}} \gamma, \eta\right]\right| \ dy',
\end{split}
\end{equation*}
where $A$ is the projection of $\mathcal{B}_{ft}^+$ on the $(N-1)$-coordinates determinated by $y'$ and $\delta=\delta(A)$ is the $N$-th coordinate of the corrisponding point in $\mathcal{B}_{ft}^+$.
Since by \eqref{Cont_comp}, $w=0$ on $\partial \Omega' \times [t_1,t_2]$, we have that $w(\Phi(y_1,...,y_{N-1}, \bar \delta))=0$. Therefore, using the one dimensional Poincar\'e inequality we obtain
\begin{equation}\label{STIMA_Fermi}
\begin{split}
     \int_{B_{\bar \delta}^+} |w(x)|^2 \ dx &\leq \left(\frac{3}{2}\right)^{N-1} \int_{A}\delta(y')^2 \left(\int_{0}^{\delta} |\partial_{y_N}w(\Phi(y))|^2 \ dy_N\right) \left|\operatorname{det}\left[\partial_{y_1} \gamma ,...,\partial_{y_{N-1}} \gamma, \eta\right]\right| \ dy'\\
     &\leq {\bar \delta}^2\left(\frac{3}{2}\right)^{N-1} \int_{A} \left(\int_{0}^{\delta} |\nabla_y w(\Phi(y))|^2 \ dy_N\right) \left|\operatorname{det}\left[\partial_{y_1} \gamma ,...,\partial_{y_{N-1}} \gamma, \eta\right]\right| \ dy'\\
     &\leq {\bar \delta}^2 3^{N-1} \int_{\mathcal{B}_{ft}^+} |\nabla_y w(\Phi(y))|^2\left|\operatorname{det}\left[\partial_{y_1} \gamma ,...,\partial_{y_{N-1}} \gamma, \eta\right]\right| |\operatorname{det}(I - y_N \mathcal{B}) | \ dy\\
     &= {\bar \delta}^2 3^{N-1} \int_{\mathcal{B}_{ft}^+} |\nabla_y w(\Phi(y))|^2 |\operatorname{det}J_{\Phi}(y)| \ dy \leq {\bar \delta}^2C\int_{B_{\bar \delta}^+} |\nabla w|^2 \ dx,
\end{split}
\end{equation}
where $C=C(\partial C_+,N)$. Note that, in the last computation, we also used the lower bound in \eqref{det_INEQ}.\\
Let now $\overline{D_2}= (\overline{\Omega \setminus \Omega'}) \times [t_1,t_2] \subset \mathbb{R}^N$. Consider $r> 2\bar\delta$ and the following covering of $\overline{D_2}$ by $\{B_r(\bar x_j)\}_j$, where $\bar x_j \in \partial \Omega \times [t_1,t_2]$. Since $\overline{D_2}$ is compact, by definition, there exists a finite $m>0$ subcover $\{B_r(\bar x_1),...,B_r(\bar x_m)\}$ of the open covering $\{B_r(\bar x_j)\}_j$ such that:
\begin{equation*}
    \overline{D_2}\subset \bigcup_{i=1}^m B_r(\bar x_i).
\end{equation*}
As in \eqref{pallette}, consider the sets $B_{\bar \delta,i}^+$ w.r.t. the set of centers $(\bar x_i)_{i=1,\ldots,m}$.
Therefore, using \eqref{STIMA_Fermi}, we get
\begin{equation}\label{termI2}
\begin{split}
    I_2 &= \int_{D_2}  |(u - u^{\bar \lambda + \varepsilon})^+|^2 \ dx \leq \int_{\bigcup_{i=1}^m B_{\bar \delta,i}^+} |(u - u^{\bar \lambda + \varepsilon})^+|^2 \ dx\\
    &\leq \sum_{i=1}^m \int_{ B_{\bar \delta,i}^+} |(u - u^{\bar \lambda + \varepsilon})^+|^2 \ dx \leq {\bar \delta}^2C \sum_{i=1}^m\int_{B_{\bar \delta,i}^+} |\nabla (u - u^{\bar \lambda + \varepsilon})^+|^2 \ dx\\
    &\leq {\bar \delta}^2mC\int_{D_2} |\nabla (u - u^{\bar \lambda + \varepsilon})^+|^2 \ dx.
\end{split}
\end{equation}
Concerning terms $I_1$ and $I_3$, they can be treated similarly. Hence, we focus on term $I_3$. Using the one dimensional Poincar\'e inequality, we obtain
\begin{equation}\label{termI3}
\begin{split}
   I_3 = \int_{D_3}  |(u - u^{\bar \lambda + \varepsilon})^+|^2 \ dx &= \int_\Omega \left(\int_{t_2}^{\bar \lambda + \varepsilon} |(u - u^{\bar \lambda + \varepsilon})^+|^2 \ dx_N\right) \ dx'\\
   &\leq (\bar \lambda + \varepsilon- t_2)^2 \int_\Omega \left(\int_{t_2}^{\bar \lambda + \varepsilon} |\partial_{x_N}(u - u^{\bar \lambda + \varepsilon})^+|^2 \ dx_N\right) \ dx'\\
   &\leq (\bar \lambda + \varepsilon- t_2)^2 \int_\Omega \left(\int_{t_2}^{\bar \lambda + \varepsilon} |\nabla(u - u^{\bar \lambda + \varepsilon})^+|^2 \ dx_N\right) \ dx'\\
   &= (\bar \lambda + \varepsilon- t_2)^2 \int_{D_3} |\nabla(u - u^{\bar \lambda + \varepsilon})^+|^2 \ dx.
\end{split}
\end{equation}
Analogously for the term $I_1$, we have
\begin{equation}\label{termI1}
   I_1 = \int_{D_1}  |(u - u^{\bar \lambda + \varepsilon})^+|^2 \ dx\leq (t_1)^2 \int_{D_1} |\nabla(u - u^{\bar \lambda + \varepsilon})^+|^2 \ dx.
\end{equation}
Summing up, by \eqref{3pezzi}, \eqref{termI2}, \eqref{termI3} and \eqref{termI1} we deduce
\begin{equation*}
\begin{split}
 \int_{D^*}  &|(u - u^{\bar \lambda + \varepsilon})^+|^2 \ dx \\
 &\quad\leq   {\bar \delta}^2 3mC\int_{D_2} |\nabla (u - u^{\bar \lambda + \varepsilon})^+|^2 \ dx +   (\bar \lambda + \varepsilon- t_2)^2 \int_{D_3} |\nabla(u - u^{\bar \lambda + \varepsilon})^+|^2 \ dx\\
&\qquad\qquad+  (t_1)^2 \int_{D_1} |\nabla(u - u^{\bar \lambda + \varepsilon})^+|^2 \ dx\\
&\quad \leq \max\{{\bar \delta}^2 mC(\bar \lambda + \varepsilon- t_2)^2,(t_1)^2\} \int_{D^*} |\nabla(u - u^{\bar \lambda + \varepsilon})^+|^2 \ dx.
\end{split}
\end{equation*}
Therefore, by \eqref{Pre_F}, we can choose $\bar\delta, t_1,t_2$ such that
\begin{equation}\label{miserve}
  \int_{D^*} |\nabla(u - u^{\bar \lambda + \varepsilon})^+|^2 \ dx \leq 0.
\end{equation}
Hence it follows that, $\nabla (u-u^{\bar \lambda + \varepsilon})^+ = 0$ in ${D^*}$, but since $(u-u^{\bar \lambda + \varepsilon})^+ = 0$ on $\Omega_0$, then $u \leq u^{\bar \lambda + \varepsilon}$ in ${D^*}=C_{\bar \lambda + \varepsilon}\setminus K$. Moreover, by \eqref{Cont_comp}, $u < u^{\bar \lambda + \varepsilon}$ in $K$. Therefore, $u \leq u^{\bar \lambda + \varepsilon}$ in $C_{\bar \lambda + \varepsilon}$, in contradiction with the definition of $\bar \lambda$.

\subsubsection{The case \texorpdfstring{$1 < p \leq \frac{2N+2}{N+2}$}{1< p < = 2N+2/N+2}}\ \\

Let us define $Z$ as the set of critical points of $u$, namely $Z=\{x \in C_{+} : \nabla u = 0\}$. We recall that $u \leq u^{\bar \lambda}$ in $ C_{\bar \lambda}$. By the strong comparison principle, see \cite{D,P-S},  if $\mathcal{C}$ is a connected component of $ C_{\bar \lambda} \setminus Z$, then $u < u^{\bar \lambda}$ unless $u \equiv u^{\bar \lambda}$ in $\mathcal{C}$. 
\begin{center}
Consider a connected component $\mathcal{C}$ of $ C_{\bar \lambda} \setminus Z$ and suppose that $u \equiv u^{\bar \lambda}$ in $C$.\\
\end{center}
Note that $\mathcal{C}\cap\overline{\Omega_0}=\emptyset$ by Proposition \ref{N_mon}.  Let $\hat{\mathcal{C}} := \mathcal{C} \cup R_{\bar \lambda}(\mathcal{C})$, where $R_{\bar \lambda}$ stands for the reflection with respect to $\Omega \times \{\bar \lambda\}$. Since $u \equiv u^{\bar \lambda}$  in $\mathcal{C}$, by the local symmetry we have that
\begin{equation*}
\partial \hat{\mathcal{C}} \subset Z \cup \left ( \partial \Omega \times (0,2 \bar \lambda) \right ).
\end{equation*}
To proceed, let us define the function
\begin{equation*}
G_{\varepsilon}(s) :=
\begin{cases}
0, & \text{if } |s| \leq \varepsilon,\\
2s-2\varepsilon, & \text{if } \varepsilon \leq |s| \leq 2\varepsilon,\\
s, & \text{if } |s| \geq 2 \varepsilon.
\end{cases}
\end{equation*}
Passing trough a density argument, we can plug $T_\varepsilon(|\nabla u|):=\frac{G_{\varepsilon}(|\nabla u|)}{|\nabla u|}$ as test function in \eqref{main_problem}, with the notation $T_\varepsilon(0)=0$, and integrating over $\hat{\mathcal{C}}$, we get
\begin{equation*}
\int_{\hat{\mathcal{C}}} |\nabla u|^{p-2} \nabla u \cdot \nabla(T_\varepsilon(|\nabla u|)) \ dx= \int_{\hat{\mathcal{C}}} f(u)\cdot T_\varepsilon(|\nabla u|) \ dx.
\end{equation*}
Note that, since $G_\varepsilon(|\nabla u|)$ vanishes in a neighborhood of each critical point, we can use $T_\varepsilon(|\nabla u|)$ as test function.\\
Since $|f(u) \cdot T_\varepsilon(|\nabla u|)| \leq f(u) \in L^1(\hat{\mathcal{C}})$ and $f(u) \cdot T_\varepsilon(|\nabla u|) \rightarrow f(u)$ a.e. in $\hat{\mathcal{C}}$ as $\varepsilon \rightarrow 0$, by the dominated convergence theorem, recalling \eqref{Hp}, we deduce
\begin{equation}\label{positive_f}
\int_{\hat{\mathcal{C}}} f(u) \cdot T_\varepsilon(|\nabla u|) \ dx \rightarrow \int_{\hat{\mathcal{C}}} f(u) \ dx > 0,
\end{equation}
as $\varepsilon \rightarrow 0$.\\
Let us now estimate the integral on the right hand side of the previous equality, namely
\begin{equation*}
\begin{split}
&\left | \int_{\hat{\mathcal{C}}} |\nabla u|^{p-2} \nabla u \cdot \nabla(T_\varepsilon(|\nabla u|)) \ dx \right | = \left | \int_{\hat{\mathcal{C}}} |\nabla u|^{p-2} \nabla u \cdot \nabla|\nabla u| T'_\varepsilon(|\nabla u|) \ dx \right |\\ 
&\quad= \left | \int_{\varepsilon \leq |\nabla u| \leq 2\varepsilon} |\nabla u|^{p-2} \nabla u \cdot \nabla|\nabla u| T'_\varepsilon(|\nabla u|) \ dx \right | \leq 2\int_{\varepsilon \leq |\nabla u| \leq 2\varepsilon} |\nabla u|^{p-2} |D^2 u| \ dx\\
&\quad \leq \left ( \int_{\varepsilon \leq |\nabla u| \leq 2\varepsilon} |\nabla u|^{p-2-\beta} |D^2 u|^2 \ dx \right )^{\frac{1}{2}} \cdot \left ( \int_{\varepsilon \leq |\nabla u| \leq 2\varepsilon} |\nabla u|^{p-2+\beta} \ dx \right )^{\frac{1}{2}},
\end{split}
\end{equation*}
where we used the fact that $T'_\varepsilon(|\nabla u|) \neq 0$ only in $\varepsilon \leq |\nabla u| \leq 2\varepsilon$ and in this set behaves like $\frac{1}{\varepsilon}$.\\
Since $|\nabla u|^{p-2-\beta} |D^2 u|^2 \in L^1(\varepsilon \leq |\nabla u| \leq 2\varepsilon)$,  see \cite[Theorem $1.1$]{D-S-2}, for $\beta \approx 1$, we have
\begin{equation*}
\left ( \int_{\varepsilon \leq |\nabla u| \leq 2\varepsilon} |\nabla u|^{p-2-\beta} |D^2 u|^2 \ dx \right )^{\frac{1}{2}} \cdot \left ( \int_{\varepsilon \leq |\nabla u| \leq 2\varepsilon} |\nabla u|^{p-2+\beta} \ dx \right )^{\frac{1}{2}} \rightarrow 0,
\end{equation*}
as $\varepsilon \rightarrow 0$, in contradiction with \eqref{positive_f}. This is an absurd, and therefore, we get $u < u^{\bar \lambda}$ in $\mathcal{C}$.

\vspace{0.2cm}

Summing up, we proved that $u < u^{\bar \lambda}$ in any connected component of ${C}_{\bar \lambda} \setminus Z$.\\

By positivity of $f$, using \cite[Proposition 6.1]{MMS}, we have $|Z|=0$, then we can consider a set $\mathcal{Z}$
such that $Z\subset \mathcal{Z}$ and the measure $|\mathcal{Z}|$ is arbitrarily small.\\
Consider a compact set $K \subset \subset (C_{\bar \lambda} \setminus \mathcal{Z})$ such that $|{C}_{\bar \lambda} \setminus K|$ is small enough because of $|\mathcal{Z}|$ is arbitrarily small. We know that, by the previous argument, $u^{\bar \lambda} - u \geq m > 0$ in $K$, moreover, since $K$ is compact, there exists $\varepsilon > 0$ small such that $u^{\bar \lambda + \varepsilon} - u \geq \frac{m}{2}>0$ in $K$ and $|{C}_{\bar \lambda+\varepsilon} \setminus K|$ is still small enough.\\
Let us now test the equations satisfied by $u$ and $u^{\bar \lambda + \varepsilon}$ with $(u - u^{\bar \lambda + \varepsilon})^+$, obtaining
\begin{equation}\label{critico_D}
\begin{split}
\int_{{C}_{\bar \lambda+\varepsilon} \setminus K} &(|\nabla u|^{p-2} \nabla u - |\nabla u^{\bar \lambda + \varepsilon}|^{p-2} \nabla u^{\bar \lambda + \varepsilon}) \cdot \nabla (u - u^{\bar \lambda + \varepsilon})^+ \ dx \\
&= \int_{{C}_{\bar \lambda+\varepsilon} \setminus K} (f(u) - f(u^{\bar \lambda + \varepsilon}))(u- u^{\bar \lambda + \varepsilon})^+ \ dx \leq C_L \int_{{C}_{\bar \lambda+\varepsilon} \setminus K} |(u - u^{\bar \lambda + \varepsilon})^+|^2 \ dx,
\end{split}
\end{equation}
where we used the Lipschitz continuity of $f$.\\
Let us set $w(x):=(u(x)-u^{\bar \lambda + \varepsilon}(x))^+$, for $x\in {C}_{\bar \lambda+\varepsilon} \setminus K$. Let $\Omega' \subset \subset \Omega$, then we define the compact set as follows $K:= \overline{\Omega'} \times [t_1,t_2] \setminus \mathcal{Z}$. We split the integral on the right-hand side of \eqref{critico_D} in:
\begin{equation}\label{4pezzi}
\begin{split}
\int_{{C}_{\bar \lambda+\varepsilon} \setminus K} |w|^2 \ dx &= \int_{(\Omega \times (0,t_1)) \setminus \mathcal{Z}}|w|^2 \ dx + \int_{(\Omega \times (t_2,\bar \lambda+\varepsilon)) \setminus \mathcal{Z}}|w|^2 \ dx \\
&\qquad +  \int_{(\Omega \times [t_1,t_2]) \setminus (\Omega' \times [t_1,t_2] \cup \mathcal{Z})}|w|^2 \ dx + \int_{Z_\delta}|w|^2 \ dx=: I_1+I_2+I_3 + I_4.
\end{split}
\end{equation}
Reasoning as in the proof of \eqref{miserve}, we have 
\begin{equation*}
\begin{split}
    &\int_{{C}_{\bar \lambda+\varepsilon} \setminus K} |\nabla (u - u^{\bar \lambda + \varepsilon})^+|^2 \ dx \leq 0.
\end{split}
\end{equation*}
Finally, by the previous inequality and by the fact that $u < u^{\bar \lambda + \varepsilon}$ in $K$, we infer that:
\begin{equation*}
u \leq u^{\bar \lambda + \varepsilon} \quad \text{in} \quad {C}_{\bar \lambda+\varepsilon},
\end{equation*}
in contraddiction with the definition of $\bar \lambda$, proving that $\bar \lambda = + \infty$ and thus the thesis.\\

\subsection{The case \texorpdfstring{$p > 2$}{p>2}}\ \\ \ \\
The case $p>2$ requires more care. In particular, we need a suitable cube decomposition $\{A_i\}_i$ (see Section \ref{notations}) that allows us to prove a weighted Poincar\'e inequality on each element $A_i$.
By the strong comparison principle, see \cite[Theorem $3.7$]{D-S}, we have $u < u^{\bar \lambda}$ in $C_{\bar \lambda}$. Using Theorem \ref{thm:strongggggggg}, we also have that $u<u_{\bar\lambda}$ on $\partial \Omega\times(0,\bar{\lambda})$; this because, since our domain is cylindrical, each connected component of the lateral surface touches $\partial \Omega_0$ so the case $u\equiv u_{\bar{\lambda}}$ on $\partial \Omega\times(0,\bar{\lambda})$ is avoided by Dirichlet condition. 

Moreover we point out that in a neighborhood of $\overline{\Omega_0}$, $u<u_{\bar\lambda}$; indeed $u$ satisfies the Dirichlet condition on $\Omega_0$ while the reflected solution is positive in the interior of the cylinder and remains positive on the lateral surface; otherwise, the Hopf Lemma would imply that the derivative with respect to the normal is positive, in contrast with the Neumann condition.

 By the uniform continuity, now, we deduce that, for $\varepsilon>0$ small enough, $u < u^{\bar \lambda+\varepsilon}$ in the compact set, 
 $$
  K=\overline{\Omega}\times[0,t_2], \qquad t_2\simeq\bar{\lambda}.
 $$
 Next the proof reduces in establishing a comparison principle in $D_3:=\Omega \times (t_2,\bar \lambda + \varepsilon)$ (see Figure \ref{fig:rettangolo-lambda}).\\
Testing our problem with $(u- u^{\bar \lambda+ \varepsilon})^+$ we have
\begin{equation}\label{D_NUOVO}
    \begin{split}
\int_{D_3} &(|\nabla u|^{p-2} \nabla u - |\nabla u^{\bar \lambda + \varepsilon}|^{p-2} \nabla u^{\bar \lambda + \varepsilon}) \cdot \nabla (u - u^{\bar \lambda + \varepsilon})^+ \ dx\\
&\qquad\qquad\qquad\quad= \int_{D_3} (f(u) - f(u^{\bar \lambda + \varepsilon}))(u- u^{\bar \lambda+ \varepsilon})^+ \ dx,
\end{split}
\end{equation}
being $(u- u^{\bar \lambda+ \varepsilon})^+\equiv 0$ on $K$. 
Thus, we have
\begin{equation}\label{sinistra_D1}
\begin{split}
\int_{D_3} &(|\nabla u|^{p-2} \nabla u - |\nabla u^{\bar \lambda + \varepsilon}|^{p-2} \nabla u^{\bar \lambda + \varepsilon})  \cdot \nabla (u - u^{\bar \lambda + \varepsilon})^+ \ dx\\
&\qquad\geq C(p)\int_{D_3} (|\nabla u| + |\nabla u^{\bar \lambda  + \varepsilon}|)^{p-2} |\nabla (u - u^{\bar \lambda + \varepsilon})^+|^2 \ dx,
\end{split}
\end{equation}
where $C$ is a positive constant depending only on $p$.\\
Moreover, using the local Lipschitz continuity of $f$ we get:
\begin{equation}\label{estimate}
\begin{split}
\int_{D} &(f(u) - f(u^{\bar \lambda  + \varepsilon}))(u- u^{\bar\lambda  + \varepsilon})^+ \ dx\\
&\leq C_L \int_{D_3} |(u- u^{\bar \lambda + \varepsilon})^+|^2 \ dx\\
\end{split}
\end{equation}
As remarked before, a crucial step, when $p\geq 2$, is the establishment of a weighted Poincar\'e inequality. The proof is neither easy nor straightforward. This difficulty arises from the fact that our problem involves both Dirichlet and Neumann boundary conditions. As will be shown below, in order to prove the required Poincar\'e inequality, we need to work with subdomains of the cylinder $ C_+$ that are convex and whose ratio between the Lebesgue measure and the diameter remains constant as the Lebesgue measure of the domain tends to zero. This makes the proof particularly technical and involved.\\
Recall that $\Omega\subset \mathbb{R}^{N-1}$  is a smooth domain; by assumption, the boundary $\partial \Omega$ is a smooth manifold of dimension $N-2$, 
therefore, for each point $p \in \partial \Omega$ there exist a radius $r_p>0$, 
a ball $B_{r_p}(p) \subset \mathbb{R}^{N-1}$, an open set $V_p \subset \mathbb{R}^{N-1}$, 
and a smooth diffeomorphism $\Phi_p$ such that
\[
\Phi_p : B_{r_p}(p) \to \Phi_p(B_{r_p}(p)) = V_p,
\]
and
\[
\Phi_p\bigl(B_{r_p}(p) \cap \partial \Omega\bigr)
= \Phi_p(B_{r_p}(p)) \cap \bigl(\mathbb{R}^{N-2} \times \{0\}\bigr),
\]
where $\{0\} \in \mathbb{R}$. 
Since $\partial \Omega$ is compact, there exists a finite collection of points 
\begin{equation}\label{centri}
p_i \in \partial \Omega,\quad i=1,2,\dots,m,
\end{equation} 
such that
\[
\partial \Omega \subset \bigcup_{i=1}^m B_{r_{p_i}}(p_i).
\]
For each point $p \in \partial \Omega$ there exists an index $i \in \{1,2,\dots,m\}$ such that
$p \in B_{r_{p_i}}(p_i)$. Then we choose $r_p>0$ sufficiently small so that
\[
B_{r_p}(p) \subset B_{r_{p_i}}(p_i).
\]
Since $\partial \Omega$ is compact, there exists $\hat r>0$ such that
\begin{equation}\label{eq:min1}
\forall p \in \partial \Omega \ \exists\, i \in \{1,2,\dots,m\} \text{ such that }
B_{\hat r}(p) \subset B_{r_{p_i}}(p_i).
\end{equation}
Let us consider the center $g_j$ of each cube $I^\partial_{\delta,j}$ defined in the Section \ref{notations}. 
Let us consider its projection, say $p_{g_j}$, onto the boundary $\partial\Omega$. 
This is possible being the domain $\Omega$ smooth, 
and therefore there exists a neighborhood $U$ of $\partial \Omega$ with the unique nearest point property, 
together with a smooth projection map $P : U \to \partial\Omega$.

Now, let us choose $\delta>0$ sufficiently small such that the cube 
\begin{equation}\label{eq:min2}
Q^G_{3\delta,j}\subset B_{\frac{\hat r}{2}}(p_{g_j}).
\end{equation}
Summarizing let us consider the sets $A_k$ of Theorem \ref{main_CUBE} such that
$\overline{A_k} \cap \partial \Omega \neq \emptyset$. Then, for any $k$, there exist a point $p_i \in \partial \Omega$ in\eqref{centri} such that
\[
A_k \subset Q^G_{3\delta,j} \subset B_{\frac{\hat r}{2}}(p_{g_j}) \subset B_{r_{p_i}}(p_i),
\]
by the constructions in \eqref{eq:min1} and \eqref{eq:min2}. Therefore, using the diffeomorphism $\Phi_{p_i}$, we obtain
\[
A_k^* = \Phi_{p_i}(A_k),
\]
which satisfies
\begin{equation}\label{eq:domtrasf}
\operatorname{diam}(A_k^*) \sim \delta, \qquad |A_k^*| \sim \delta^{N-1},
\end{equation}
by Theorem \ref{main_CUBE} and the fact that $\Phi_{p_i}$ is a diffeomorphism.

We now define the set
\[
A_{k,c}^* := \bigcap \bigl\{ C \subset \mathbb{R}^{N-1} \,\big|\, 
C \text{ is convex and } A_k^* \subset C \bigr\}.
\]
By definition, $A_{k,c}^*$ is the smallest convex set containing $A_k^*$, that is,
the convex hull of $A_k^*$.
Let
\[
A_{k,c} = \Phi_{p_i}^{-1}(A_{k,c}^*),
\]
and define the set $B_{k,c} \subset \mathbb{R}^N$ by
\[
B_{k,c} =\bigcup_{l\geq 1} A_{k,c} \times [t_2+(l-1)\delta, t_2+l\delta]=:\bigcup_{l\geq 1}B_{k,l,c},
\]
where $t_2$ will be chosen later (see Figure \ref{fig:rettangolo-lambda}).

We define the change of variables
\[
\Psi_{p_i} : B_{r_{p_i}}(p_i) \times [h, h+\delta] \to V_{p_i} \times [h, h+\delta], h\geq 0
\]
as
\begin{equation*}
\begin{cases}
y' = \Phi_{p_i}(x'), \\
y_N = x_N.
\end{cases}
\end{equation*}
Finally, we consider the sets
\[
B_{k,l,c}^* = \Psi_{p_i}(B_{k,l,c}).
\]
Since $A_{k,c}^*$ is convex and $ [t_2+(l-1)\delta, t_2+l\delta]$ is an interval, it follows that
$B_{k,l,c}^*$ is a convex set and satisfies
\begin{equation}\label{eq:diamarea}
\operatorname{diam}(B_{k,l,c}^*) \sim \delta, \qquad |B_{k,l,c}^*| \sim \delta^{N}.
\end{equation}
Therefore, by construction (see Theorem \ref{main_CUBE}), since the sets $B_{k,c}$ and
$$
B^\circ_{k,c}:=\bigcup_{l\geq 1} A_{k} \times [t_2+(l-1)\delta, t_2+l\delta]=:\bigcup_{l\geq 1}B^\circ_{k,l,c},
$$ 
(where $A_k$ are such that $\overline{A_k}\subset \Omega$) form a covering of $D_3$, we have
\begin{equation}\label{D_1_Monty}
\int_{D_3} |(u - u^{\bar \lambda + \varepsilon})^+|^2 \, dx
\leq \sum_{k=1}^s \int_{B_{k,c}} |(u - u^{\bar \lambda + \varepsilon})^+|^2 \, dx + \sum_{k=1}^{\tilde N} \int_{B^\circ_{k,c}} |(u - u^{\bar \lambda + \varepsilon})^+|^2 \, dx.
\end{equation}
Let us now deal with the first integral on the r.h.s. of \eqref{D_1_Monty}. Using the change of variables defined above, we obtain
\begin{eqnarray}\label{eq:partenza}
&&\int_{B_{k,l,c}} |(u - u^{\bar \lambda + \varepsilon})^+|^2 (x)\, dx \\\nonumber
&&\quad = \int_{B_{k,l,c}^*} |(u - u^{\bar \lambda + \varepsilon})^+|^2 (\Psi^{-1}_{p_i}(y)) \,
\left| \det J_{\Psi^{-1}_{p_i}}(y) \right| \, dy \label{teofon-1}\\\nonumber
&&\quad \leq C(\Psi^{-1}_{p_i}) \int_{B_{k,l,c}^*} |(u - u^{\bar \lambda + \varepsilon})^+|^2 (\Psi^{-1}_{p_i}(y)) \, dy.
\end{eqnarray}
In order to recover the required weighted Poincar\'e inequality, we define
\begin{equation*}
w(y)=
\begin{cases}
\left(u-u^{\bar \lambda + \varepsilon}\right)^+(\Psi^{-1}_{p_i}(y',y_N)) 
& \text{if } (y',y_N) \in B_{k,l,c}^*,\\[6pt]
-\left(u-u^{\bar \lambda + \varepsilon}\right)^+(\Psi^{-1}_{p_i}(y',-y_N)) 
& \text{if } (y',y_N) \in B_{k,l,c}^{*,r},
\end{cases}
\end{equation*}
where $(y',y_N) \in B_{k,l,c}^{*,r}$ if and only if $(y',-y_N) \in B_{k,l,c}^*$.\\
Since $w$ has zero mean on $B_{k,l,c}^* \cup B_{k,l,c}^{*,r}$, for a.e.\ $y \in B_{k,l,c}^*$ we have (see \cite[Lemma~7.16]{GT})
\begin{equation}\label{teofon}
\begin{split}
|w(y)| &\leq \hat C \int_{B_{k,l,c}^* \cup B_{k,l,c}^{*,r}} 
\frac{|\nabla w(z)|}{|y-z|^{N-1}} \, dz \\
&= \hat C \int_{B_{k,l,c}^*} \frac{|\nabla w(z)|}{|y-z|^{N-1}} \, dz
 + \hat C \int_{B_{k,l,c}^{*,r}} \frac{|\nabla w(z)|}{|y-z|^{N-1}} \, dz \\
&\leq 2 \hat C \int_{B_{k,l,c}^*} \frac{|\nabla w(z)|}{|y-z|^{N-1}} \, dz.
\end{split}
\end{equation}
Here
\[
\hat C = \frac{(\operatorname{diam}(B_{k,l,c}^* \cup B_{k,l,c}^{*,r}))^N}{N \mathcal{L}(B_{k,l,c}^* \cup B_{k,l,c}^{*,r})} = \bar C(N).
\]
We point out that this constant depends only on $N$, thanks to \eqref{eq:diamarea}, see also Remark \ref{rem1}. The purpose of considering this cube decomposition (see Theorem \ref{main_CUBE}) is to apply the weighted Poincar\'e inequality within each element of the family ${\{B_{k,c}^*\}}_k$, exploiting the fact that the constant $\hat C$ in \eqref{teofon-1} depends only on $N$ and not on the diameter or the size of the sets. Indeed, later on it will be necessary to consider domains that become arbitrarily small while retaining control of the constants.\\
Moreover, in order to apply inequality~\eqref{teofon-1}, we need to work on convex domains (see \cite[Lemma~7.16]{GT}). This justifies the delicate construction of the sets $B_{k,c}^*$, which was necessary to ensure that the underlying domains are convex.\\
Therefore, after deriving \eqref{teofon-1}, we can apply \cite[Theorem $8$]{FMS} (in particular \cite[Corollary $2$]{FMS}), to deduce the weighted Poincar\'e inequality with weight $\rho=|\nabla u|^{p-2}$, obtaining:
\begin{eqnarray}\label{eq:arrivo}
&&\int_{B_{k,l,c}^*} |(u- u^{\bar \lambda  + \varepsilon})^+|^2 (\Psi^{-1}_{p_i}(y)) \ dy \\\nonumber
&&\quad\leq \hat C^2 C_p(B_{k,l,c}^*)  \int_{B_{k,l,c}^*}|\nabla_y u|^{p-2}(\Psi^{-1}_{p_i}(y)) |\nabla_y(u- u^{\bar \lambda + \varepsilon})^+|^2(\Psi^{-1}_{p_i}(y))\, dy.
\end{eqnarray}
Using the change of variables $\Psi_{p_i}$ in \eqref{eq:arrivo}, from \eqref{eq:partenza} and since $p>2$, we obtain
\begin{equation}\label{termini_inv}
\begin{split}
\int_{B_{k,l,c}} &|(u- u^{\bar \lambda + \varepsilon})^+|^2 (x)\, dx  \\
&\leq C(\Psi_{p_i})\, \hat C^2\, C_p(B_{k,l,c}^*)  
\int_{B_{k,l,c}} |\nabla_x u|^{p-2}(x)\,
|\nabla_x (u- u^{\bar \lambda + \varepsilon})^+|^2(x)\, dx\\
&\leq C(\Psi_{p_i})\, \hat C^2\, C_p(B_{k,l,c}^*)  
\int_{B_{k,l,c}} (|\nabla_x u|+|\nabla_x u^{\bar \lambda + \varepsilon}|)^{p-2}(x)\,
|\nabla_x (u- u^{\bar \lambda + \varepsilon})^+|^2(x)\, dx.
\end{split}
\end{equation}
We remark that the constant $C_p(B_{k,l,c}^*) \to 0$ as $\delta \to 0$.
Indeed, if $\delta \to 0$ then $|B_{k,l,c}| \to 0$ by Theorem~\ref{main_CUBE}.
The same holds for the Lebesgue measure of the transformed set $B_{k,l,c}^*$, see
\eqref{eq:domtrasf} and \eqref{eq:diamarea}.
Therefore, $C_p(B_{k,l,c}^*) \to 0$ as $\delta \to 0$, see \cite[Corollary $2$]{FMS}.\\
Concerning the terms involving the integrals on $B^\circ_{k,l,c}$, a similar argument implies that
\begin{equation*}
\int_{B^\circ_{k,l,c}} |(u- u^{\bar \lambda + \varepsilon})^+|^2\, dx\leq C(\delta)  
\int_{B^\circ_{k,l,c}} (|\nabla u|+|\nabla u^{\bar \lambda + \varepsilon}|)^{p-2}\,
|\nabla (u- u^{\bar \lambda + \varepsilon})^+|^2\, dx,
\end{equation*}
where $C(\delta) \rightarrow 0$ as $\delta \rightarrow 0$.\\
Moreover, using the previous inequality and \eqref{termini_inv}, we can estimate the r.h.s. of \eqref{D_1_Monty} as follows
\begin{equation}\label{D_1_Monty_good}
\begin{split}
&\int_{D_3} |(u - u^{\bar \lambda + \varepsilon})^+|^2 \, dx\\
&\quad \leq C(N)\sup_{k=1,...,s}{C_p(B_{k,c}^*)}\sum_{k=1}^s\int_{B_{k,c}} (|\nabla u|+|\nabla u^{\bar \lambda + \varepsilon}|)^{p-2}\,
|\nabla (u- u^{\bar \lambda + \varepsilon})^+|^2\, dx\\
& \quad\qquad + C(N)C(\delta)  
\sum_{k=1}^{\tilde N}\int_{B^\circ_{k,c}} (|\nabla u|+|\nabla u^{\bar \lambda + \varepsilon}|)^{p-2}\,
|\nabla (u- u^{\bar \lambda + \varepsilon})^+|^2\, dx\\
&\quad \leq C(N)\max\{s\sup_{k=1,...,s}{C_p(B_{k,c}^*)},\tilde{N}C(\delta)\}\int_{D_3} (|\nabla u|+|\nabla u^{\bar \lambda + \varepsilon}|)^{p-2}\,
|\nabla (u- u^{\bar \lambda + \varepsilon})^+|^2\, dx\\
&\quad = \tilde C_1(\delta) \int_{D_3} (|\nabla u|+|\nabla u^{\bar \lambda + \varepsilon}|)^{p-2}\,
|\nabla (u- u^{\bar \lambda + \varepsilon})^+|^2\, dx,
\end{split}
\end{equation}
where $ \tilde C_1(\delta) \rightarrow 0$, as $\delta \rightarrow 0$.\\

Therefore, by \eqref{D_NUOVO}, \eqref{sinistra_D1}, \eqref{estimate}, \eqref{D_1_Monty_good}, we conclude that 
\begin{equation*}
\left(1-C(p) \tilde C_1(\delta)\right)\int_{D_3} (|\nabla u| + |\nabla u^{\bar \lambda  + \varepsilon}|)^{p-2} |\nabla (u - u^{\bar \lambda + \varepsilon})^+|^2 \ dx \leq 0.
\end{equation*}
Thus, for $\varepsilon>0$ small and $t_2$ close to $\bar{\lambda}$, $\delta$ can be choosen small such that
\begin{equation*}
    \left(1-C(p) \tilde C_1(\delta)\right)>0,
\end{equation*}
obtaining that $u \leq u^{\bar \lambda + \varepsilon}$ in $D_3$. Moreover, we know that $u < u^{\bar \lambda + \varepsilon}$ in $K$ and therefore we conclude that $u \leq u^{\bar \lambda + \varepsilon}$ in $C_{\bar \lambda + \varepsilon}$, in contradiction with the definition of $\bar \lambda$.

Therefore, we proved that for any $p > 1$, any solution $u$ of \eqref{main_problem} is monotone increasing w.r.t. the $x_N$ direction, namely
\begin{equation}\label{FINISH}
\frac{\partial u}{\partial x_N} \geq 0 \quad \text{in } C_+.
\end{equation}
In particular, by the strong comparison principle for the derivatives of the solutions in \cite{D-S}, \eqref{FINISH} holds with the strict inequality if $p > \frac{2N+2}{N+2}$.
\end{proof}
\section{Classification and Liouville type results}\label{liouville}
We prove here the classification results for the classical Allen-Cahn type problem. The existence of a 1-D solution is easy to verify and already known. The focus is on proving that this is the only possible solution. The monotonicity result comes into play,
at first allowing to deduce the limiting profile of the solution, that is actually achieved uniformly. This will allow the application of a refined version of the sweeping principle. 
\subsection{Allen-Cahn type problems}\label{ACtype}\ \\
Here we study bounded solutions of Allen-Cahn type equations, namely, we deal with 
\begin{equation}\label{ACequation}
\begin{cases}
-\Delta_{p} u = u(1-u^2)^{p-1} & \text{in } C_+\\
0< u < 1  & \text{in } C_+\\
\partial_{\nu} u = 0  & \text{on } \partial{\Omega} \times (0,+\infty)\\
u=0 & \text{on } \Omega \times \{0\}.\\
\end{cases}
\end{equation}

A first crucial information concerns the limiting profile at infinity that we deduce in the following Lemma:
\begin{lemma}\label{u_n_1}
Let $u \in C^{1,\alpha}_{loc}(\overline{C}_+)$ be a weak solution of \eqref{ACequation}. Then
\begin{equation*}
\lim_{x_N \rightarrow +\infty} u(x',x_N) = 1, \quad \text{uniformly in } x'.
\end{equation*}
\end{lemma}
\begin{proof} Since the solution $u$ is $x_N$-monotone by Theorem  \ref{Monotonicity}, and recalling the notation $x=(x',x_N)$, for any $x\in C_+$, for any $n\in \mathbb{N}$, define
\begin{equation*} 
u_n(x):= u(x',x_N+n).
\end{equation*}
and
\begin{equation*}
w(x'):= \lim_{n \rightarrow \infty} u_n(x).
\end{equation*}

The boundedness of $u$ guarantees, for any $x' \in \Omega$, that the limit is finite.
Consider as test function,
\begin{equation*}
\varphi = \varphi_1(x')\varphi_2(x_N), \quad \text{with } \varphi_1 \in C^{\infty}(\overline \Omega) \ \text{and } \varphi_2 \in C^\infty_c(\mathbb{R}^+).
\end{equation*}
In addition, we choose $\varphi_2$ such that,
\begin{equation}\label{int_2}
\int_{\mathbb{R}^+} \varphi_2 \ d x_N= 1.
\end{equation}
The function $u_n$, satisfies
\begin{equation}\label{ACequationq}
\begin{cases}
-\Delta_{p} u = u(1-u^2)^{p-1} & \text{in } C_+\\
0< u < 1  & \text{in } C_+\\
\partial_{\nu} u = 0  & \text{on } \partial{\Omega} \times (0,+\infty).
\end{cases}
\end{equation}
Testing the equation in \eqref{ACequationq} with $\varphi$, we get:
\begin{equation*}
\int_{C_+} |\nabla u_n|^{p-2} \nabla u_n \cdot \nabla \varphi \ dx= \int_{C_+} u_n(1-u_n^2)^{p-1} \varphi \ dx.
\end{equation*}
Exploiting the definition of $\varphi$, we get
\begin{equation}\label{wf_ac_phi}
\begin{split}
\int_{C_+}\varphi_2 |\nabla u_n|^{p-2} \nabla u_n \cdot \nabla \varphi_1  \ dx &+ \int_{C_+} \varphi_1|\nabla u_n|^{p-2} \nabla u_n \cdot \nabla \varphi_2  \ dx\\
&= \int_{C_+} u_n(1-u_n^2)^{p-1} \varphi_1\varphi_2 \ dx.
\end{split}
\end{equation}
Using the standard regularity results \cite{DB,Lib0}, since $u_n$ is bounded, we have that 
$$
\|u_n\|_{C^{1,\alpha}_{loc}(\overline{\Omega}\times\mathbb{R}^+)}<1,
$$
hence, by Ascoli-Arzel\'a Theorem, passing to the limit in \eqref{wf_ac_phi} as $n \rightarrow \infty$, we obtain:
\begin{equation*}
\begin{split}
\int_{C_+} \varphi_2 |\nabla w|^{p-2} \nabla w \cdot \nabla \varphi_1 \ dx &+ \int_{C_+} \varphi_1  |\nabla w|^{p-2} \nabla w \cdot \nabla \varphi_2\ dx\\
&= \int_{C_+} w(1-w^2)^{p-1} \varphi_1\varphi_2 \ dx.
\end{split}
\end{equation*}
Moreover, since $\nabla w = (\nabla_{x'} w, 0)$ and $\nabla \varphi_2 = (0', \partial_{x_N}\varphi_2)$,
\begin{equation*}
\int_{C_+} \varphi_1|\nabla w|^{p-2} \nabla w \cdot \nabla \varphi_2  \ dx = 0.
\end{equation*}
Hence, exploiting Fubini's Theorem and by \eqref{int_2}, namely $\int_{\mathbb{R}^+} \varphi_2 \ d x_N= 1$, we end up with:
\begin{equation}\label{fdeb}
\int_{\Omega} |\nabla w|^{p-2} \nabla w \cdot \nabla \varphi_1 \ dx'= \int_{\Omega} w(1-w^2)^{p-1} \varphi_1 \ dx', \quad \forall \varphi_1 \in C^\infty(\overline \Omega).
\end{equation}
In addition, exploiting $C^1$-convergence, we infer that $\partial_\nu w = 0$ on $\partial \Omega$, therefore, $w$ weakly solves
\begin{equation*}
\begin{cases}
-\Delta_p w = w(1-w^2)^{p-1} & \text{in } \Omega\\
0\leq w \leq 1 & \text{in } \Omega\\
\partial_\nu w=0 & \text{on } \partial \Omega.
\end{cases}
\end{equation*}
We claim now that $w \equiv 1$. To prove the claim, first of all, we notice that $w > 0$ as a consequence of the monotonicity of $u$ w.r.t. the $x_N$-direction and the strong maximum principle.\\
Furthermore, choosing $\varphi=1$ in $\overline{\Omega}$, as a test function in \eqref{fdeb}, we get
\begin{equation*}
\int_{\Omega} w(1-w^2)^{p-1} \ dx' = 0.
\end{equation*}
Hence, since $w > 0$ in $\Omega$, the claim is proved. The thesis follows exploiting the $C^1$-convergence of the sequence $\{u_n\}$.

\end{proof}
The next lemma, whose proof is standard and therefore omitted, establishes the existence and uniqueness of monotone increasing solutions of the one-dimensional Allen-Cahn type problem associated to \eqref{ACequation}.
\begin{lemma}\label{MI_AC}
There exists a unique monotone increasing function
\begin{equation*}
v(t) = \tanh\left(\frac{t}{(2p-2)^{\frac{1}{p}}}\right),
\end{equation*}
solution of
\begin{equation}\label{AC_1}
\begin{cases}
-\left[(v')^{p-1}\right]' = v(1-v^2)^{p-1} & \text{in } \mathbb{R}^+,\\
v(0)=0; \ \displaystyle  \lim_{t \rightarrow +\infty} v(t)=1.
\end{cases}
\end{equation}
\end{lemma}

The final preliminary result is a comparison principle between two solutions $u,v$ of problem \eqref{ACequation} in the domain $C_+$.
We want to point out that that there exists $0<\eta_1<1$ such that $g'(t)=(t(1-t^2)^{p-1})'\leq -D<0$ (when $\eta_1<t<1$). Moreover, by Lemma \ref{u_n_1}, there exists 
\begin{equation}\label{teta}
\theta=\theta(\eta_1)>0,
\end{equation} 
such that if $u,v$ are weak solutions to \eqref{ACtype}, $u,v>\eta_1$ for all $C_+\ni x=(x',x_N)$ with $x_N>\theta$. 
\begin{lemma}\label{confrontopiccolo}
Let $u,v$ be $C^{1,\alpha}_{loc}(\overline{C}_+)$-solutions of \eqref{ACequation} and let $\theta$ be defined in \eqref{teta}.
Suppose that
$$u\geq v \mbox{  in } \Omega\times[0,\theta].$$ 
Then $u\geq v$ in $C_+$.
\end{lemma}
\begin{proof}
Let $R>\theta$ and let $\varphi_R(x_N)$ be a standard cut-off function such that $\varphi_R \equiv 1$ if $x_N\in [\theta, R]$, $\varphi_R \equiv 0$ if $x_N\geq 2R$ and $|\varphi_R'|\leq \frac2R$ if $x_N\in (R,2R)$.\\
For $p\geq 2$, let us consider $(v-u)^+\varphi^2_R$ as test function in \eqref{ACequation}. We obtain
\begin{equation*}
\begin{split}
  \int_{[v \geq u]} |\nabla u|^{p-2} \nabla u \cdot \nabla(v-u)^+\varphi^2_R \ dx&+2 \int_{[v \geq u]} |\nabla u|^{p-2} \nabla u\cdot \nabla \varphi_R (v-u)^+\varphi_R  \ dx\\
&= \int_{[v \geq u]} u(1-u^2)^{p-1} (v-u)^+\varphi^2_R \ dx,
\end{split}
\end{equation*}
and
\begin{equation*}
\begin{split}
  \int_{[v \geq u]} |\nabla v|^{p-2} \nabla v \cdot \nabla(v-u)^+\varphi^2_R \ dx&+2 \int_{[v \geq u]} |\nabla v|^{p-2} \nabla v\cdot \nabla \varphi_R (v-u)^+\varphi_R  \ dx\\
&= \int_{[v \geq u]} v(1-v^2)^{p-1} (v-u)^+\varphi^2_R\ dx.
\end{split}
  \end{equation*}
Taking the difference between the previous two equations and by standard inequalities (see \cite{D,LI}), we obtain
  \begin{equation*}
\begin{split}
  c_1\int_{[v \geq u]}& (|\nabla u|+|\nabla v|)^{p-2} |\nabla (v-u)^+|^2 \varphi^2_R \ dx\\ 
&\leq 2c_2\int_{[v \geq u]} (|\nabla v|+|\nabla u|)^{p-2}|\nabla (v-u)^+|  (v-u)^+|\nabla\varphi_R|\varphi_R  \ dx\\
 &\qquad+ \int_{[v \geq u]} (v(1-v^2)^{p-1}-u(1-u^2)^{p-1}) (v-u)^+\varphi^2_R \ dx.
\end{split}
 \end{equation*}  
Then  using a standard Young inequality, we get
  \begin{equation*}
\begin{split}
  c_1\int_{[v \geq u]} &(|\nabla u|+|\nabla v|)^{p-2} |\nabla (v-u)^+|^2 \varphi^2_R \ dx\\
  &\leq {2c_2\sigma} \int_{[v \geq u]} (|\nabla v|+|\nabla u|)^{p-2}|\nabla (u-v)^+|^2 \varphi_R^2dx\\
  &\qquad+\frac{c_2}{2\sigma R^2} \int_{[v \geq u]} (|\nabla v|+|\nabla u|)^{p-2}((v-u)^+)^2 \ dx\\
  &\qquad+ \int_{[v \geq u]} -D ((v-u)^+)^2\varphi^2_R \ dx.
\end{split}
 \end{equation*}
Let $C:=C(c_2,p, \|\nabla u\|_\infty,\|\nabla v\|_\infty)$, then the previous inequality becomes
\begin{equation*}
\begin{split}
  c_1\int_{[v \geq u]} &(|\nabla u|+|\nabla v|)^{p-2} |\nabla (v-u)^+|^2 \varphi^2_R \ dx\\
&\leq {2c_2\sigma} \int_{[v \geq u]} (|\nabla v|+|\nabla u|)^{p-2}|\nabla (u-v)|^2 \varphi_R^2dx\\
&\qquad+\left(\frac{C}{\sigma R^2}-D\right) \int_{[v \geq u]}((v-u)^+)^2 \ dx.
\end{split}
 \end{equation*}
Next, we choose $\sigma$ small enough so that
$$
 c_1-2c_2\sigma>0,
$$
and then $R$ large such that $\frac{C}{\sigma R^2}-D<0$.\\
Therefore, we arrive to
$$
 \int_{[v \geq u]} (|\nabla u|+|\nabla v|)^{p-2} |\nabla (v-u)^+|^2 \varphi^2_R \ dx<0,
$$
and this implies that $u\geq v$.

The case $p\in(1,2)$ is similar, but it requires more care.  Let us consider a parameter $\alpha\geq 1$ and the function $\phi$ defined as
$$
 \phi:=((v-u)^+)^\alpha \varphi_R^{\alpha+1}.
$$
Testing our problem \eqref{ACequation} with $\phi$, we have
\begin{equation*}
\begin{split}
\alpha\int_{[v \geq u]} &(|\nabla v|^{p-2} \nabla v-|\nabla u|^{p-2} \nabla u) \cdot \nabla(v-u)^+((v-u)^+)^{\alpha-1}\varphi^{\alpha+1}_R \ dx\\
&+(\alpha+1) \int_{[v \geq u]}(|\nabla v|^{p-2} \nabla v-|\nabla u|^{p-2} \nabla u) \cdot \nabla \varphi_R ((v-u)^+\varphi_R)^{\alpha}  \ dx\\
&\qquad= \int_{[v \geq u]} (v(1-v^2)^{p-1}-u(1-u^2)^{p-1}) \phi \ dx.
\end{split}
\end{equation*}
By standard inequalities (see \cite{D,LI}), we arrive to:
\begin{equation*}
\begin{split}
\alpha c_1 \int_{[v \geq u]}& (|\nabla u|+|\nabla v|)^{p-2} |\nabla (v-u)^+|^2 ((v-u)^+)^{\alpha-1}\varphi^{\alpha+1}_R \ dx\\
&\leq (\alpha+1) \int_{[v \geq u]} |\nabla (u-v)|^{p-1}  |\nabla \varphi_R| ((v-u)^+\varphi_R)^{\alpha}  \ dx \\
&\qquad+ \int_{[v \geq u]} (v(1-v^2)^{p-1}-u(1-u^2)^{p-1}) \phi \ dx.
\end{split}
\end{equation*}
Applying Young’s inequality, with exponents $(\alpha+1,\frac{\alpha+1}{\alpha})$, to the first integral on the right-hand side of the previous equation, we obtain
\begin{equation*}
\begin{split}
 \alpha c_1 \int_{[v \geq u]} &(|\nabla u|+|\nabla v|)^{p-2} |\nabla (v-u)^+|^2 ((v-u)^+)^{\alpha-1}\varphi^{\alpha+1}_R \ dx\\
&\leq \alpha C \sigma^{\frac{\alpha+1}{\alpha}}\int_{[v \geq u]} ((v-u)^+\varphi_R)^{\alpha+1}  \ dx+\frac{C}{\sigma^{\alpha+1}} \int_{[v \geq u]}  |\nabla \varphi_R|^{\alpha+1} \ dx\\
&\qquad -D \int_{[v \geq u]} ((v-u)^+)^{\alpha+1} \varphi_R^{\alpha+1} \ dx,
\end{split}
\end{equation*}
where $C:=C(c_2,p, \alpha, \|\nabla u\|_\infty,\|\nabla v\|_\infty)$ is a positive constant.\\
Let us now choose $\sigma$ small enough in such a way that
$$
 \alpha C \sigma^{\frac{\alpha+1}{\alpha}}-D<0,
$$
thus, we have
\begin{eqnarray*}
 \int_{[v \geq u]} (|\nabla u|+|\nabla v|)^{p-2} |\nabla (v-u)^+|^2 ((v-u)^+)^{\alpha-1}\varphi^{\alpha+1}_R \ dx \leq\frac{RC}{\alpha c_1 \sigma^{\alpha+1}R^{\alpha+1}}.
 \end{eqnarray*}
Letting $R\to+\infty$,  by Fatou's Lemma,  we get $u\geq v$.
\end{proof}
\noindent We are now in position to prove Theorem \ref{AC_teo}.
\begin{proof}[Proof of Theorem \ref{AC_teo}]
First of all, we define the translation of the solution $u$ w.r.t. the $x_N$ direction, namely
\begin{equation*}
u_\tau = u(x',x_N+\tau), \quad \text{with } \tau \geq 0.
\end{equation*}
Consider $v$ the one-dimensional solution of \eqref{AC_1}, given by Lemma \ref{MI_AC}; let $\theta$ in \eqref{teta} such that $v\geq \eta_1$ on $[\theta,+\infty)$ and  $v\leq \eta_1$ on $[0,\theta]$. Exploiting Lemma \ref{u_n_1}, let us choose $\tau > 0$ in such a way that $u_\tau > \eta_1$ in $\Omega\times(0,\tau)$. Hence
\begin{equation*}
u_\tau(x',x_N) > v(x',x_N) \quad \text{in } \Omega\times(0,\tau).
\end{equation*}
Notice that, such a choice is justified by the fact that $u$ and $v$ are monotone increasing in the $x_N$ direction and converge uniformly to $1$ as $x_N \rightarrow \infty$.\\
By Lemma \ref{confrontopiccolo}, it follows that $u_\tau \geq v$ in $C_+$. Let,
\begin{equation*}
\bar \tau = \inf\{\tau > 0 : u_\tau \geq v \ \text{in } C_+\}.
\end{equation*}
and suppose, by contradiction, that $\bar \tau > 0$.\\
\textbf{Claim:} $u_{\bar \tau} > v$ in $\overline \Omega \times [0,\theta]$.\\
To prove the claim we argue by contradiction. Suppose that there exists $\bar x=(\bar x',\bar x_N) \in \overline \Omega \times [0,\theta]$ such that $u_{\bar \tau}(\bar x) = v(\bar x)$. 
We distinguish three cases:
\begin{enumerate}
\item[(i)] If $\bar x \in \overline \Omega_0$, by the definition of $u_{\bar \tau}$ and since $u>0$ inside the cylinder $C_+$, we have that $u_{\bar\tau}(x',0) > v(x',0)=0$, for any $x' \in \overline \Omega$.
\item[ii)] If $\bar x \in \Omega \times (0,\theta]$, since by the definition of $\bar \tau$,
\begin{equation*}
u_{\bar \tau} \geq v \quad \text{in } C_+,
\end{equation*}
then, by strong comparison principle (see \cite{D,P-S}), $u_{\bar \tau} \equiv v$ in $\Omega \times (0,\theta+\delta)$, ($\delta>0)$, a contradiction since $u_{\bar \tau}(x',0) > v(x',0)=0$, for any $x' \in \Omega$.
\item[iii)] If instead, $\bar x \in \partial \Omega \times (0,\theta]$, note that if $u_{\bar \tau}(\bar x) = v(\bar x)$, then, by Fermat Theorem, this implies that $|\partial_{x_N} u_{\bar \tau}(\bar x)|=|\partial_{x_N} v(\bar x)|=v'(x_N) >0$ and then, $\nabla u_{\bar \tau}(\bar x) \neq 0$ and $\nabla v(\bar x) \neq 0$. Therefore, standard regularity results imply that the solutions are smooth in a neighborhood of $\bar{x}$.
Then, \cite[Theorem $2.7.1$]{P-S} implies
\begin{equation*}
 \partial_\nu u_{\bar \tau} > \partial_\nu v \quad \text{at } \bar x.
\end{equation*}
However, this is a contradiction by the homogeneous-Neumann boundary conditions on $\partial \Omega\times (0,+\infty)$.
\end{enumerate}
Hence, the claim follows.\\ 
Finally, since $u_{\bar \tau}- v \geq \delta > 0$ and $\overline \Omega \times [0,\theta]$ is compact, by the uniform continuity we have:
\begin{equation*}
u_{\bar \tau - \varepsilon} - v \geq \frac{\delta}{2} > 0, \quad \text{in } \overline \Omega \times [0,\theta], \ \text{for } \varepsilon > 0 \text{ small enough}.
\end{equation*}
Exploiting again Lemma \ref{confrontopiccolo}, we conclude that $u_{\bar \tau - \varepsilon} \geq v$ in $\Omega \times \mathbb{R}^+$, in contradiction with the definition of $\bar \tau$. Thus, $\bar \tau = 0$ and, in particular, $u \geq v$ in $C_+$.\\
In order to prove that $v \geq u$ in $C_+$ an analogous idea can be used. 

Hence, summing up, we proved that $u(x',x_N) = v(x',x_N)=v(x_N)$ in $\Omega \times \mathbb{R}^+$. Thus, thanks to Lemma \ref{MI_AC}, any solution $u$ of \eqref{ACequation} is one-dimensional, namely:
\begin{equation*}
u(x',x_N)=u(x_N)= \tanh\left(\frac{x_N}{(2p-2)^{\frac{1}{p}}}\right).
\end{equation*}
\end{proof}

\subsection{Stability and Liouville type results}\ \\
Relying on the results established in Theorem \ref{Monotonicity}, this subsection is dedicated to the proof of Theorem \ref{Stabilità}.\\
In the following we set $p\geq 2$. The linearized operator of \eqref{main_problem} is given by
\begin{equation}\label{eq:linearized}
\begin{split}
L_u(v,\varphi)&=\int_{C_+}|\nabla u|^{p-2}(\nabla v, \nabla \varphi)\, dx +(p-2)\int_{C_+}|\nabla u |^{p-4}(\nabla u, \nabla v)(\nabla u, \nabla \varphi)\, dx\\
&\qquad\qquad\qquad-\int_{C_+} f'(u)v\varphi\, dx,
\end{split}
\end{equation}
for every $\varphi:=\varphi|_{C_+}$ such that $\varphi \in C^\infty_c(\mathbb{R}^N_+)$.
\begin{definition}\label{def:stab1}
We say that a solution of \eqref{main_problem} is stable if 
\begin{equation}\label{eq:stable}
L_u(\varphi, \varphi)\geq 0,\qquad \forall \varphi\in C^\infty_c(\mathbb{R}^N_+).
\end{equation}
\end{definition}
\begin{theorem}
Let $u$ be a solution of \eqref{main_problem} and assume that $\partial_{x_N}u>0$ in $\Omega \times (0, +\infty)$. Then $u$ is stable.
\end{theorem}
\begin{proof}
In the following let $\varepsilon>0$ and let  us denote $u_{x_N}=\partial_{x_N}u$. We define  the following test function
\[\varphi=\frac{\psi^2}{\varepsilon + u_{x_N}},\]
for every $\psi:=\psi|_{C_+}$ such that $\psi \in C^\infty_c(\mathbb{R}^N_+)$.
If we consider $v=u_{x_N}$ in \eqref{eq:linearized}, we have that
\begin{equation}\label{eq:linzer}
L_u(u_{x_N}, \varphi)=0.
\end{equation}
This fact involves some technical arguments, so we provide some details.
The main idea is to  take a smooth test function $\varphi \in C^\infty_c(\mathbb{R}^N_+)$, and then, consider  $\varphi_{x_N}$ as test function in \eqref{main_problem}. 
 By \cite{MMS}, we know that 
 \[|\nabla u|^{p-2}u_{x_N}\in W^{1,2}(\Omega\times (0,t)), \quad \text{for }t>0.\]
Integration by parts then yields \eqref{eq:linzer}.\\
Therefore,
\begin{eqnarray}\label{eq:exp1}
\nonumber 0=L_u(u_{x_N},\varphi)&=&-\int_{C_+}\frac{|\nabla u|^{p-2}}{({\varepsilon + u_{x_N})^2}}|\nabla u_{x_N}|^2\psi^2\, dx-(p-2)\int_{C_+}\frac{|\nabla u|^{p-4}}{({\varepsilon + u_{x_N})^2}}(\nabla u, \nabla u_{x_N})^2\psi^2\, dx\\\nonumber
&&+2\int_{C_+}|\nabla u|^{p-2}(\nabla u_{x_N},\nabla \psi) \frac{\psi}{\varepsilon + u_{x_N}}\, dx\\
&&+2(p-2)\int_{C_+}|\nabla u|^{p-4}(\nabla u, \nabla u_{x_N})(\nabla u, \nabla \psi)\frac{\psi}{\varepsilon+u_{x_N}}\, dx\\\nonumber
&&-\int_{C_+}f'(u)u_{x_N}\frac{\psi^2}{\varepsilon + u_{x_N}}\, dx.
\end{eqnarray}
The classical Young inequality implies 
\begin{eqnarray*}
&&2\int_{C_+}|\nabla u|^{p-2}(\nabla u_{x_N},\nabla \psi) \frac{\psi}{\varepsilon + u_{x_N}}\, dx\\\nonumber
&&\qquad\leq\int_{C_+}\frac{|\nabla u|^{p-2}}{({\varepsilon + u_{x_N})^2}}|\nabla u_{x_N}|^2\psi^2\, dx+\int_{C_+}|\nabla u|^{p-2}|\nabla \psi|^2\, dx,
\end{eqnarray*}
and
\begin{eqnarray*}
&&2(p-2)\int_{C_+}|\nabla u|^{p-4}(\nabla u, \nabla u_{x_N})(\nabla u, \nabla \psi)\frac{\psi}{\varepsilon+u_{x_N}}\, dx\\\nonumber
&&\qquad\leq(p-2)\int_{C_+}\frac{|\nabla u|^{p-4}}{({\varepsilon + u_{x_N})^2}}(\nabla u, \nabla u_{x_N})^2\psi^2\, dx\\\nonumber
&&\qquad \qquad+(p-2)\int_{C_+}|\nabla u|^{p-4}|\nabla u|^2(\nabla u,\nabla \psi )^2\, dx.
\end{eqnarray*}
Therefore, from \eqref{eq:exp1} we deduce
\begin{equation}\label{eq:exp2}
0\leq\int_{C_+}|\nabla u|^{p-2}|\nabla \psi|^2\, dx
+(p-2)\int_{C_+}|\nabla u|^{p-4}(\nabla u,\nabla \psi )^2\, dx-\int_{C_+}f'(u)u_{x_N}\frac{\psi^2}{\varepsilon + u_{x_N}}\, dx.
\end{equation}  
Letting $\varepsilon\rightarrow 0$, by Fatou Lemma, we obtain the thesis \eqref{eq:stable}.
\end{proof}
The following proposition provides a control on the decay of our solution in the cylinder $C_+$. To prove the nonexistence result stated in Theorem \ref{Stabilità}, we need to define the problem in the entire cylinder $$C:=\Omega \times (-\infty, +\infty).$$ For this reason, we prove the following proposition in the whole cylinder.
\begin{proposition}\label{pro:stab}
Let $ u\in C_{loc}^{1,\alpha}(\overline C)$ be a solution to 
\begin{equation}\label{eq:solcylinder}
\begin{cases}
-\Delta_{p} u = | u|^{q-1} u &  \text{in } C\\
\partial_{\nu}  u = 0  & \text{on } \partial{C},  
\end{cases}
\end{equation}
where $p\geq 2$ and $q>p-1$. Let us assume that $u$ is stable, that is 
\begin{equation}\label{eq:stab2}
\int_{C}|\nabla u|^{p-2}|\nabla \varphi|^2\, dx +(p-2)\int_{C}|\nabla u |^{p-4}(\nabla u, \nabla \varphi)^2\, dx \geq q\int_{C}|u|^{q-1} \varphi^2\, dx,
\end{equation}
for all $\varphi \in C^\infty_c(\mathbb{R}^N)$.
Then for every integer $k$ satisfying 
\begin{equation}\label{eq:fra1}
k\geq \max\left\{\frac{q+1}{q-(p-1)}; 2\right\},
\end{equation}
there exists a positive constant $C=C(k,p,q)$ such that 
\begin{equation}\label{eq:P145}
\int_{C}|\nabla  u|^p\psi^{pk}+| u|^{q+1}\psi^{pk}\, dx\leq C\int_{C}|\nabla \psi|^{p\frac{q+1}{q-(p-1)}}\, dx,\end{equation}
for every $\psi \in C^\infty_c(\mathbb{R}^N)$, with $0\leq \psi\leq 1$.
\end{proposition}
\begin{proof}
The proof of this proposition follows the approach of \cite{DFSV}. Here we provide  a general overview.

{\bf Step 1.} We claim that for any  
$0<\varepsilon<1$, there exists a positive constant $C_\varepsilon$ (depending on $\varepsilon$ and $p$) such  that
\begin{equation}\label{eq:P141}
(1 -\varepsilon^{2})\int_{C}|\nabla u|^p\varphi^p\, dx \leq \int_{C}|u|^{q+1}\varphi^p\, dx+C_\varepsilon\int_{C}|\nabla \varphi|^p |u|^{p}\, dx,
\end{equation}
for any $\varphi \in C^\infty_c(\mathbb{R}^N)$.
This step follows using $\xi=u\varphi^p$ as test function in \eqref{eq:solcylinder} and by the Young inequality.

{\bf Step 2.} Set 
\[\alpha_\varepsilon=\frac{q}{p-1}-\frac{1+\varepsilon^2}{1-\varepsilon^2}.\]
Then, there exists $\beta=\beta(p,q,\varepsilon)$ such that 
\begin{equation}\label{eq:P142} 
\alpha\int_{C}|u|^{q+1}\varphi^p\, dx\leq \beta \int_{C}|\nabla \varphi|^p|u|^{p}\, dx,
\end{equation}
for any $\varphi \in C^\infty_c(\mathbb{R}^N)$.\ \\
The idea to prove this step is to exploit \eqref{eq:linearized} and \eqref{eq:stable}. Indeed, since by hypothesis  $u$ is a stable solution of \eqref{eq:solcylinder} we deduce that
\[q\int_{C}|u|^{q-1}\varphi^2\, dx \leq (p-1)\int_{C}|\nabla u|^{p-2}|\nabla \varphi|^2\, dx.\]
Using in the above inequality the test function $u\varphi^{\frac p2}$ and then the Young inequality (after some computations), we get \eqref{eq:P142}. 

{\bf Step 3.} There exists a constant $C=C(k,p,q)>0$ such that 
\begin{equation}\label{eq:P143} 
\int_{C}|u|^{q+1}\psi^{pk}\, dx \leq C\int_{C}|\nabla \psi|^{p\frac{q+1}{q-(p-1)}} \, dx,
\end{equation}
for any $\psi\in C^\infty_c(\mathbb{R}^N)$ with $0\leq\psi\leq 1$.\\
Inequality \eqref{eq:P143} follows using $\varphi=\psi^k $ with $\psi\in C^\infty_c(\mathbb{R}^N)$ and $0\leq\psi\leq 1$, as test function in \eqref{eq:P142} and by assumption \eqref{eq:fra1}.

{\bf Step 4.} There exists a constant $C=C(k,p,q)>0$ such that 
\begin{equation}\label{eq:P144}
\int_{C}|\nabla u|^p\psi^{pk}\, dx\leq C\int_{C}|\nabla \psi|^{p\frac{q+1}{q-(p-1)}}\, dx,
\end{equation} 
for any $\psi \in C^\infty_c(\mathbb{R}^N)$ with $0\leq\psi\leq 1$.\\
Combining \eqref{eq:P141} and \eqref{eq:P142}, we deduce that
\[\int_{C}|\nabla u|^p\varphi^p\, dx \leq \bar C\int_{C}|\nabla \varphi|^p|u|^{p}\, dx,\]
where 
$\bar C_\varepsilon=\frac{C_\varepsilon}{1-\varepsilon^2}+\frac{\beta}{\alpha(1-\varepsilon^2)}.$ Using $\varphi=\psi^k $, with $\psi\in C^\infty_c(\mathbb{R}^N)$ and $0\leq\psi\leq 1$, as test function in the previous inequality, we obtain
\begin{equation*}
\begin{split}
\int_{C}|\nabla u|^p\psi^{pk}\,dx&\leq C\int_{C} \psi^{p(k-1)}|\nabla \psi|^p|u|^{p}\, dx\\\nonumber
&\leq\left(\int_{C}\left(|u|^{p}\psi^{p(k-1)}\right)^\frac{q+1}{p}\,dx\right)^\frac{p}{q+1}\left(\int_{C}\left(|\nabla \psi|^{p\frac{q+1}{q-(p-1)}}\right)\, dx\right)^\frac{q-(p-1)}{q+1}.
\end{split}
\end{equation*}
Finally \eqref{eq:P144} follows using \eqref{eq:P143}  and the fact that 
\[(k-1)\left(q+1\right)\geq pk,\]
since we assumed $k\geq \frac{q+1}{q-(p-1)}$, see \eqref{eq:fra1}.

The proof concludes by combining equations \eqref{eq:P143} and \eqref{eq:P144}.
\end{proof}

We are now in position to prove Theorem \ref{Stabilità}.

\begin{proof}[Proof of Theorem \ref{Stabilità}]
Let $u$ be a solution of \eqref{main_problem_S}. Define $\tilde u: C\rightarrow R$ as
\[\tilde u:=
\begin{cases} u(x', x_N) & \text{if } x_N\geq 0\\
-u(x', -x_N) & \text{if } x_N<0.
 \end{cases}\]
Then we deduce that $\tilde u$ is a solution to  \eqref{eq:solcylinder}. Moreover, exploiting the fact that $u$ is stable in the sense of Definition \ref{def:stab1}, we deduce \eqref{eq:stab2}.\\
This theorem now follows as a corollary of Proposition \ref{pro:stab}, once we show that the right-hand side of \eqref{eq:P145} vanishes. 
To see this, let us define the test function  
$\psi\in C_c^{\infty}(\mathbb{R}^N)$, $\psi(x)=\varphi(x_N)$, where $\varphi (s)\in C_c^{\infty}(\mathbb{R})$ such that
\[\varphi (s)=\begin{cases}
1, & \text{if } |s|\leq R,\\
0, & \text{if } |s|\geq  2R,\\
\left |\partial_s \varphi(s)\right |\leq \frac 2R, & \text{if } R\leq |s| \leq 2 R.
\end{cases}
\]
Using this test function  in \eqref{eq:P145}, we obtain 
\begin{equation}\label{eq:fra2}
\int_{C_+}|\nabla u|^p\psi^{pk}\, dx+|u|^{q+1}\psi^{pk}\, dx\leq CR^{1-p\frac{q+1}{q-(p-1)}}.
\end{equation}
We observe that 
\[{1-p\frac{q+1}{q-(p-1)}}<0,\]
since the function 
$g(s)=\frac{s+1}{s-(p-1)}$ strictly decreasing and satisfies
\[\lim_{s\rightarrow (p-1)^+} g(s)=+\infty\quad \text{and}\quad  \lim_{s\rightarrow+\infty} g(s)=1.\]
The conclusion now follows by applying Fatou's Lemma in \eqref{eq:fra2}.

\end{proof}
{\bf Acknowledgements} All the authors are partially supported also by Gruppo Nazionale per l’Analisi Matematica, la Probabilit\'a e le loro Applicazioni (GNAMPA) of
the Istituto Nazionale di Alta Matematica (INdAM).\\

{\bf Data Availability Statement} All data generated or analyzed during this study are included in this published article.
\end{document}